\documentclass[12pt]{amsart}
\usepackage{amssymb, tensor, fullpage, textcomp, pb-diagram, wasysym, wrapfig, comment, xspace, leftindex, xcolor, appendix, tensor, relsize}
\usepackage[Symbol]{upgreek}
\usepackage[mathscr]{euscript}
\usepackage[shortlabels]{enumitem}

\let\oldtocsection=\tocsection

\let\oldtocsubsection=\tocsubsection

\let\oldtocsubsubsection=\tocsubsubsection

\renewcommand{\tocsection}[2]{\hspace{0em}\oldtocsection{#1}{#2}}
\renewcommand{\tocsubsection}[2]{\hspace{1em}\oldtocsubsection{#1}{#2}}
\renewcommand{\tocsubsubsection}[2]{\hspace{2em}\oldtocsubsubsection{#1}{#2}}

\DeclareFontFamily{U}{fsy}{}
\DeclareFontShape{U}{fsy}{m}{n}{<->s*[.9]psyr}{}
\DeclareSymbolFont{der@m}{U}{fsy}{m}{n}
\DeclareMathSymbol{\der}{\mathord}{der@m}{182}

\newcommand{\ncdf}{\mathrm{NCD}}

\newcommand{\pusix}{\R\langle\langle\varepsilon\rangle\rangle}
\newcommand{\tdense}{T_\mathrm{dense}}

\newcommand{\dkmin}{D^\kappa(T)_{\_}}

\newcommand{\dkap}{D^\kappa(T)}
\newcommand{\sdkap}{D^\kappa(T^*)}

\newcommand{\kalg}{K^{\mathrm{alg}}}

\newcommand{\dcf}{\mathrm{DCF}}

\newcommand{\D}{\mathbb{D}}
\newcommand{\scfpe}{\mathrm{SCF}_{p,e}}
\newcommand{\acvfp}{\mathrm{ACVF}_p}
\newcommand{\scvfpe}{\mathrm{SCVF}_{p, e}}
\newcommand{\nset}{\{1,\ldots,n\}}
\newcommand{\gik}{(g_i)_{i < \kappa}}

\newcommand{\fik}{(f_i)_{i < \kappa}}
\newcommand{\pik}{(\uppi_i)_{i < \kappa}}

\newcommand{\dik}{(\der_i)_{i < \kappa}}

\newcommand{\T}{\mathbb{T}}

\DeclareSymbolFont{imag@m}{OT1}{cmr}{m}{ui}
\DeclareMathSymbol{\imag}{\mathord}{imag@m}{105}

\newtheorem{theorem}{Theorem}[section]
\newtheorem*{theorem*}{Theorem}

\newtheorem*{qst*}{Question}
\newtheorem{thm}[theorem]{Theorem}
\newtheorem{proposition}[theorem]{Proposition}
\newtheorem*{proposition*}{Proposition}

\newtheorem*{Claim*}{Claim}
\newtheorem{fact}[theorem]{Fact}

\newtheorem{lemma}[theorem]{Lemma}
\newtheorem{corollary}[theorem]{Corollary}

\theoremstyle{definition}

\newtheorem{definition}[theorem]{Definition}

\theoremstyle{remark}

\newcommand{\acf}{\mathrm{ACF}}
\newcommand{\mfrak}{\mathfrak{m}}

\newcommand{\rgoup}{(\R;+,<)}

\newcommand{\rfield}{(\R;+,\times)}
\newcommand{\rcf}{\mathrm{RCF}}

\newcommand{\doag}{\mathrm{DOAG}}

\newcommand{\vvec}{\mathrm{Vec}}
\newcommand{\triv}{\mathrm{Triv}}

\newcommand{\Th}{\mathrm{Th}}
\ProvideTextCommandDefault{\cprime}{(U+042C)}

\newcommand{\Fraisse}{Fra\"iss\'e\xspace}

\newcommand{\st}{\operatorname{st}}

\newenvironment{claimproof}[1][\proofname]
               {
                 \proof[#1]
                 
               }
               {
                 \endproof
               }

\newcommand{\Sh}[1]{\ensuremath{\mathscr{#1}^{\mathrm{Sh}}}}

\newcommand{\nip}{\mathrm{NIP}}
\newcommand{\ip}{\mathrm{IP}}

\newcommand{\Cal}[1]{\ensuremath{\mathcal{#1}}}
\newcommand{\Sa}[1]{\ensuremath{\mathscr{#1}}}

\newcommand{\Z}{\mathbb{Z}}
\newcommand{\N}{\mathbb{N}}

\newcommand{\Q}{\mathbb{Q}}
\newcommand{\R}{\mathbb{R}}
\newcommand{\F}{\mathbb{F}}
\newcommand{\E}{\mathbb{E}}

\begin{document}
\title[]{Trace definability III:\\ Infinite dimensional space over a model of $T$}

\author{Erik Walsberg}
\email{erik.walsberg@gmail.com}
\begin{abstract}
We show that for a number of theories $T^*$ of model-theoretic interest there is a simpler theory $T$ and $\kappa \ge \aleph_0$ such that $T^*$ is trace equivalent to the theory of $\kappa$-dimensional space over a model of $T$.
\end{abstract}

\maketitle
\section*{Introduction}
Let $T$ be a (complete, consistent) first order theory, $\Sa M$ be a model of $T$, and $\kappa$ be an infinite cardinal.
We are concerned with the $\kappa$-dimensional space $M^\kappa$ over $\Sa M$.
We associate to this a two-sorted structure $(M^\kappa, \Sa M, \pik)$, where each $\uppi_i$ is the projection $M^\kappa \to M$ onto the $i$th coordinate.
Let $\dkap$ be the theory of such structures.
We consider the relation between this theory and two natural model-theoretic notions of reducibility between structures/theories introduced in \cite{trace1, trace2}.
They are analogous to Turing or Borel reducibility, or various other notions of reducibility considered by logicians.
We show that $T \mapsto D^\kappa(T)$ reduces the second ``local" reducibility notion to the first ``global" notion.
In turn, this produces a third sort of reducibility between theories.
This third notion allows us to give a precise meaning to the following analogy.
\[
\text{$\acf_0$ : $\dcf^{\aleph_0}$ :: $\acf_p$ : $\scfpe$}
\]
We now recall our notions of reducibility.
A structure $\Sa N$ {\bf trace defines} $\Sa M$ if there is an injection $\uptau \colon M \hookrightarrow N^n$ for some $n \ge 1$ such that every $\Sa M$-definable subset of every $M^m$ is the pullback of an $\Sa N$-definable set via $\uptau$ and $\Sa N$ {\bf locally trace defines} $\Sa M$ if for any $\Sa M$-definable sets $X_1, \ldots, X_k$ there is an injection $\uptau \colon M \hookrightarrow N^n$ for some $n \ge 1$ such that each $X_i$ is the pullback of an $\Sa N$-definable set via $\uptau$.
A theory (locally) trace defines another if every model of the second is (locally) trace definable in a model of the first, two theories are (locally) trace equivalent if each (locally) trace defines the other, and two structures are (locally) trace equivalent when their theories are.
Classification-theoretic properties such as stability, $\nip$, $k$-$\nip$, superstability, total transcendence, strong dependence, finiteness of Morley rank, finiteness of dp-rank, etc, are preserved under trace definability.
In fact, properties that are preserved under (local) trace definability can often be characterized in terms of (local) trace definability.
For example a theory $T$ is stable if and only if $T$ does not trace define $(\Q; <)$ if and only if $T$ does not locally trace define $(\Q ; <)$.
(Note that it follows that many other  model-theoretic properties, such as $\mathrm{NSOP}$, are {\it not} preserved under trace definability.)
Likewise, $T$ is $k$-$\nip$ if and only if $T$ does not trace define the generic $(k + 1)$-hypergraph if and only if $T$ does not locally trace define the generic $(k + 1)$-hypergraph.

\medskip
Local trace definability is clearly a coarser relation than trace definability, but it also reduces to trace definability.
Suppose that $T, T^*$ are theories in languages of cardinality at most $\kappa$.
We show that $T$ locally trace defines $T^*$ if and only if $\dkap$ trace defines $T^*$ if and only if $\dkap$ trace defines $\sdkap$.
We give examples of theories $T'$ which are trace equivalent to $\dkap$ for a simpler theory $T$ in a language of cardinality $\le \kappa = |T'|$.
We view this as another way of reducing $T'$ to $T$.
In this case $T$ locally trace defines a theory $T^*$ in a language of cardinality $\le \kappa$ if and only if $T'$ trace defines $T^*$.
We now give some examples.
\begin{itemize}[leftmargin=*]
\item The theory $\scfpe$ of separably closed fields of Ershov invariant $e \in \N_{\ge 1}$ and characteristic $p$  is trace equivalent to $D^{\aleph_0}(\acf_p)$.
\item The theory of a differentially closed field of characteristic zero with $\kappa$ commuting derivations is trace equivalent to $D^\kappa(\acf_0)$.
The theory of a differentially closed field of characteristic zero  with $\lambda \ge 2$ non-commuting derivations is trace equivalent to $D^{\lambda + \aleph_0}(\acf_0)$.
\item The theory of the ordered differential field of transseries and the theory of tame pairs of real closed fields are both trace equivalent to $D^{\aleph_0}(\rcf)$.
\item If $T$ is the theory of a non-algebraically closed characteristic zero $\nip$ henselian field then the relative model companion of the theory of a model of $T$ equipped with $\lambda \ge 1$ commuting derivations is trace equivalent to $D^{\lambda + \aleph_0}(T)$.
In particular the model companion of the theory of an ordered field equipped with $\lambda \ge 1$ commuting derivations is trace equivalent to $D^{\lambda + \aleph_0}(\rcf)$.
An analogue of this also holds for o-minimal theories expanding $\rcf$.
\item Let $\triv$ be the theory of an infinite set with equality.
The theory of a set equipped with $\kappa$ crosscutting equivalence relations is trace equivalent to $D^\kappa(\triv)$.
The model companion of the theory of a set equipped with $\lambda \ge 2$ unary functions is trace equivalent to $D^{\lambda + \aleph_0}(\triv)$.
The theory of the free J\'onsson-Tarski algebra is trace equivalent to $D^{\aleph_0}(\triv)$.
\item Let $\F$ be a field and $\vvec_\F$ be the theory of infinite $\F$-vector spaces.
The model companion of the theory of an $\F$-vector space equipped with $\kappa$ commuting endomorphisms is trace equivalent to $D^\kappa(\vvec_\F)$.
The model companion of the theory of an $\F$-vector space equipped with $\lambda \ge 2$ endomorphisms is trace equivalent to $D^{\lambda + \aleph_0}(\vvec_\F)$.
\item A structure is locally trace definable in $\vvec_\F$ if and only if it is trace definable in a module over an $\F$-algebra if and only if it is locally trace definable in a module over an $\F$-algebra.
If $\F$ has positive characteristic then a structure is locally trace definable in $\vvec_\F$ if and only if it is trace definable in a one-based expansion of an $\F$-vector space if and only if it locally trace definable in a one-based expansion of an $\F$-vector space.
\end{itemize}

We will also use $\dkap$ as a tool to give results about trace definability.
For example, we show that the theory of tame pairs of real closed fields trace defines the theory of dense pairs of real closed fields, but not vice versa.

\medskip
We finally explain why the connection between local trace definability and $\dkap$ is natural.
Let $I$ be an arbitrary index set and $x = (x_i)_{i \in I}$ be an $I$-tuple of variables.
We say that a subset of $M^I$ is $\Sa M$-definable if it is of the form $\{ a \in M^I : \Sa M \models \varphi(a) \}$ for an $L(M)$-formula $\varphi(x)$ with free variables from the $x_i$.
Of course, when $I = \nset$ this is the usual notion of an $\Sa M$-definable set.
The infinite case is equally natural as the boolean algebra of $\Sa M$-definable subsets of $M^I$ is isomorphic to the Lindenbaum algebra of $L(M)$-formulas with variables from the $x_i$.
We let $nI$ be the disjoint union of $n$ copies of $I$ and identify $M^{nI}$ with $(M^I)^n$.
We showed in \cite[Prop.~2.4]{trace2} that $\Sa O$ is locally trace definable in $\Sa M$ if and only if there is an injection $\uptau \colon O \hookrightarrow M^I$ for some $I$ such that every $\Sa O$-definable subset of every $O^m$ is the pullback of an $\Sa M$-definable subset of $M^{mI}$ via $\uptau$.
(This also follows from the proof of Proposition~\ref{prop:prekey} below.)
Let $(M^\kappa, \Sa M, \pik)$ be as above.
Any $\Sa M$-definable subset of $(M^\kappa)^n$ is clearly definable (in the usual sense) in $(M^\kappa, \Sa M, \pik)$.
It follows that any structure locally trace definable in $\Sa M$ is trace definable in $(M^\kappa, \Sa M, \pik)$ for sufficiently large $\kappa$.
Moreover, Proposition~\ref{prop:pik1}(4) below shows that $\dkap$ admits quantifier elimination relative to $T$.
It follows that a subset of $(M^\kappa)^n$ is definable in $(M^\kappa, \Sa M, \pik)$ if and only if it is a boolean combination of sets definable in the language of equality and $\Sa M$-definable subsets of $(M^\kappa)^n$.
It easily follows that $(M^\kappa, \Sa M, \pik)$ is locally trace equivalent to $\Sa M$.


\medskip
This research was funded in part by the Austrian Science Fund (FWF) 10.55776/PAT1673125.

\section{Conventions and background}\label{section:c & b}

\subsection{Conventions}\label{section:conventions}
Throughout $m, n, k, l, d$ are natural numbers (including $0$) and $\kappa$ is a cardinal.
All languages, structures, and theories are first order and all theories are consistent, complete, and deductively closed unless stated otherwise.
These assumptions ensure that any $L$-theory $T$ has cardinality $|L| + \aleph_0$, so if $\kappa$ is an infinite cardinal then $|T| \le \kappa$ if and only if $|L| \le \kappa$.
Throughout ``definable" without modification means ``first order definable, possibly with parameters".
We generally assume that all structures are infinite and all theories have infinite models.
Given a language $L$, structure $\Sa M$, and $A \subseteq M$, we let $L(A)$ be the expansion of $L$ by constant symbols defining the elements of $A$.

\medskip
Given a set $X$ and function $g \colon X \to X$ we let $g^{(n)} \colon X \to X$ be the $n$-fold compositional iterate of $g$ for each $n$.
So $g^{(0)}$ is the identity on $X$ and we have $g^{(n)} = g \circ g^{(n - 1)}$ for all $n \ge 1$.

\medskip
Two structures on a common domain are {\bf interdefinable} if they define the same sets.
Two structures on possibly different domains are {\bf bidefinable} if they are interdefinable up to isomorphism.
Two theories are {\bf definitionally equivalent} if they are the same up to ``change of language", see \cite[\S~2.6.2]{Hodges} for a precise definition.

\medskip
Given a structure $\Sa M$ and $A \subseteq M^m$, the structure {\bf induced} on $A$ by $\Sa M$ is the structure with an $n$-ary relation defining $Y \cap A^n$ for every $\Sa M$-definable $Y \subseteq M^{nm}$.

\medskip
A theory is {\bf geometric} if it eliminates $\exists^\infty$ and algebraic closure satisfies exchange.


\subsection{Background}\label{section:bckgrnd}
We recall background on (local) trace definability from \cite{trace1, trace2}.
For the benefit of the reader we also provide brief sketches of the proofs.

\begin{definition}\label{def:trace}
We say that $\Sa M$ is {\bf trace definable} in $\Sa N$ if one of the following equivalent conditions holds.
\begin{enumerate}[leftmargin=*]
\item Up to isomorphism, $M$ is a subset of some $N^n$ and every $\Sa M$-definable subset of every $M^m$ is of the form\footnote{Think of $Y \cap M^m$ as the ``trace" of $Y$ on $M$ and $\Sa M$ as ``definable via traces" in $\Sa N$.} $Y \cap M^m$ for some $\Sa N$-definable $Y \subseteq N^{nm}$.
\item There is an injection $\uptau \colon M \hookrightarrow N^n$ for some $n$ such that every $\Sa M$-definable subset of every $M^m$ is of the form
\[
\{ (a_1, \ldots, a_m) \in M^m : (\uptau(a_1), \ldots, \uptau(a_m)) \in Y \}
\]
for some $\Sa N$-definable $Y \subseteq N^{mn}$.
In this case we say that $\Sa N$ trace defines $\Sa M$ via $\uptau$.
\item There are functions $\uptau_1, \ldots, \uptau_n \colon M \to N$ such that every $\Sa M$-definable subset of every $M^m$ is of the form 
\[
\{ (a_1, \ldots, a_m) \in M^m : (\uptau_{i_1}(a_{j_1}), \ldots, \uptau_{i_d}(a_{j_d}) ) \in Y\}
\]
for some $i_1, \ldots, i_d \in \{1, \ldots, n\}$, $j_1, \ldots, j_d \in \{1, \ldots, m\}$, and $\Sa N$-definable $Y \subseteq N^d$.
\item If $L$ is a relational language then any $\Sa M$-definable $L$-structure embeds into an $\Sa N$-definable $L$-structure.
\end{enumerate}  
\end{definition}
To see that (1) implies (2) let $\uptau$ be the inclusion $M \hookrightarrow N^n$.
For the converse implication push $\Sa M$ forward via $\uptau$.
To see that (2) implies (3) let $\uptau_1, \ldots, \uptau_n$ be the components of $\uptau$.
For the converse let $\uptau = (\uptau_1, \ldots, \uptau_n)$.
We leave (4) to the reader as we do not need it here.

\begin{definition}\label{def:loc trace}
We say that $\Sa M$ is {\bf locally trace definable} in $\Sa N$ if one of the following equivalent conditions holds.
\begin{enumerate}[leftmargin=*]
\item For any $\Sa M$-definable sets $X_1\subseteq M^{m_1},\ldots,X_k\subseteq M^{m_k}$ there is an injection $\uptau\colon M \hookrightarrow N^n$ for some $n$ and $\Sa N$-definable sets $Y_1\subseteq N^{nm_1},\ldots,Y_k\subseteq N^{nm_k}$ such that we have
\[
X_i = \{ (a_1, \ldots, a_{m_i}) \in M^{m_i} : (\uptau(a_1), \ldots, \uptau(a_{m_i})) \in Y_i \} \quad \text{for $i = 1, \ldots, k$.}
\]
\item For any $\Sa M$-definable set $X \subseteq M^m$ there is an injection $\uptau \colon M \hookrightarrow N^n$ for some $n$ and $\Sa N$-definable $Y \subseteq N^{nm}$ such that we have
\[
X = \{ (a_1, \ldots, a_m) \in M^m :  (\uptau(a_1), \ldots, \uptau(a_{m})) \in Y \}.
\]
\item There is a collection $\Cal E$ of functions $M \to N$ such that every $\Sa M$-definable  subset of every $M^m$  is of the form 
\[
\{ (a_1, \ldots, a_m) \in M^m : (\uptau_{1}(a_{j_1}), \ldots, \uptau_{d}(a_{j_d}) ) \in Y\}
\]
for some $\uptau_1, \ldots, \uptau_d \in \Cal E$, $j_1, \ldots, j_d \in \{1, \ldots, m\}$, and $\Sa N$-definable $Y \subseteq N^d$.
In this case we say that $\Cal E$ witnesses\footnote{Think of the members of $\Cal E$ as $N$-valued invariants of elements of $M$ and note that every $\Sa M$-definable relation reduces to an $\Sa N$-definable relation between the invariants.} local trace definability of $\Sa M$ in $\Sa N$.
\item If $L$ is a finite relational language then any $\Sa M$-definable $L$-structure embeds into an $\Sa N$-definable $L$-structure.
\item If $L$ is a language containing a single relation then any $\Sa M$-definable $L$-structure embeds into an $\Sa N$-definable $L$-structure.
\end{enumerate}
\end{definition}
The equivalence is proven in \cite[Prop.~2.4]{trace2}.
We sketch the proof.
It is clear that (1) implies (2).
To see that (2) implies (3) fix $\uptau_X = (\uptau^1_X, \ldots, \uptau^{n_X}_X)$ as in (2) for each $\Sa M$-definable $X$ and let $\Cal E$ be the collection of all $\uptau^i_X$.
To show that (3) implies (1) we suppose that $\Cal E$ is as in (3), let $X_1, \ldots, X_k$ be $\Sa M$-definable sets, let $\Cal E^*$ be a finite subset of $\Cal E$ that handles each $X_i$, let $\uptau_1, \ldots, \uptau_n$ be an enumeration of $\Cal E^*$, and take $\uptau = (\uptau_1, \ldots, \uptau_n)$.
To see that (4) implies (1) we fix $X_1, \ldots, X_k$ as in (1) and apply (4) to the structure with domain $M$ and an $m_i$-ary relation defining each $X_i$.
The converse implication is similar, and essentially the same argument shows that (2) and (5) are equivalent.

\medskip
If $T, T^*$ are theories then $T^*$ is (locally) trace definable in $T$ if every (equivalently: some) model of $T^*$ is (locally) trace definable in a model of $T$.
Two theories are {\bf (locally) trace equivalent} if each (locally) trace defines the other.
Two structures are (locally) trace equivalent when their theories are.
Equivalently, two structures are (locally) trace equivalent if each is (locally) trace definable in an elementary extension of the other.
It is easy to see that (local) trace definability is a transitive relation and hence (local) trace equivalence is an equivalence relation.
A theory is {\bf trace maximal} if it trace defines every structure and a structure is trace maximal when its theory is.
Equivalently, a structure is trace maximal if it is trace equivalent to $(\Z; +, \times)$.
We showed in \cite[Lemma~5.2]{trace1} that $T$ is trace maximal if and only if there is $\Sa M \models T$ and infinite $A \subseteq M^m$ for some $m$ such that every subset of every $A^n$ is of the form $Y \cap A^n$ for $\Sa M$-definable $Y \subseteq M^{nm}$.
We will not need this characterization, but the reader may find it useful to think of trace maximality as the strongest possible form of the independence property.

\begin{fact}\label{fact:interpret}
If $\Sa M$ is interpretable in $\Sa N$ then $\Sa M$ is trace definable in $\Sa N$.
\end{fact}

Fact~\ref{fact:interpret} is \cite[Prop.~1.4]{trace1}.
To prove it we suppose that $\Sa N$ interprets $\Sa M$, let $X$ be an $\Sa N$-definable set and $E$ be an $\Sa N$-definable equivalence relation on $X$ such that $\Sa M$ ``lives on" $X/E$, and take $\uptau$ to be a section of the quotient map $X \to X/E$.

\begin{fact}\label{fact:trace embedd}
Suppose that $\Sa M$ and $\Sa N$ are structures in the same language, $\Sa M$ admits quantifier elimination, and $\uptau$ is an embedding $\Sa M \hookrightarrow \Sa N$.
Then  $\Sa N$ trace defines $\Sa M$ via $\uptau$.
\end{fact}

Fact~\ref{fact:trace embedd} is \cite[Prop.~1.2]{trace1}.
It follows from the elementary fact that if $\uptau$ is an arbitrary embedding $\Sa M \hookrightarrow \Sa N$ between structures in the same language then any quantifier-free definable subset of $M^m$ is the pullback via $\uptau$ of a quantifier-free definable subset of $N^m$.

\medskip
Lemma~\ref{lem:ka} allows us to keep track of the number of functions needed to locally trace define a structure.
We say that $\Sa N$ locally trace defines $\Sa M$ via $\kappa$ functions if local trace definability of $\Sa M$ in $\Sa N$ can be witnessed by a collection of $\le \kappa$ functions.

\begin{lemma}\label{lem:ka}
\hspace{.000000000000000001cm}
\begin{enumerate}[leftmargin=*]
\item Let $\Sa N$ and $\Sa M \models T$ be structures and suppose that $\Sa N$ locally trace defines $\Sa M$.
Then this can be witnessed by $|T|$ functions.
\item Suppose that some model of $T$ is locally trace definable in a model of $T^*$ via $\kappa$ functions.
Then every model of $T$ is locally trace definable in a model of $T^*$ via $\kappa$ functions.
\item If $\Sa M_1$ is locally trace definable in $\Sa M_2$ via $\kappa$ functions and $\Sa M_2$ is locally trace definable in $\Sa M_3$ via $\eta$ functions, then $\Sa M_1$ is locally trace definable in $\Sa M_3$ via $\kappa\eta$ functions.
\end{enumerate}
\end{lemma}

Hence we say that $T^*$ locally trace defines $T$ via $\kappa$ functions if some model of $T^*$ trace defines a model of $T$ via $\kappa$ functions.

\begin{proof}
(1): Let $\Cal E$ witness local trace definability of $\Sa M$ in $\Sa N$.
For every subset $X$ which is definable without parameters in $\Sa M$ let $\Cal E_X$ be a finite subset of $\Cal E$ such that $X$ is of the form $$\{ (a_1, \ldots, a_m) \in M^m : (\uptau_{1}(a_{j_1}), \ldots, \uptau_{d}(a_{j_d}) ) \in Y\}$$
for some $\uptau_1, \ldots, \uptau_d \in \Cal E_X$, $j_1, \ldots, j_d \in \{1, \ldots, m\}$, and $\Sa N$-definable $Y \subseteq N^d$.
Let $\Cal E^*$ be the union of the $\Cal E_X$.
Then $|\Cal E^*| \le |T|$ and it is easy to see that $\Cal E^*$ witnesses local trace definability of $\Sa M$ in $\Sa N$.

\medskip
(2): Suppose that $\Cal E$ is a collection of at most $\kappa$ functions witnessing local trace definability of $\Sa M \models T$ in $\Sa N \models T^*$.
Let $\Sa O$ be an arbitrary model of $T$.
Let $(\Sa M^*, \Sa N^*, \Cal E^*)$ be an $|O|^+$-saturated elementary extension of the two-sorted structure $(\Sa M, \Sa N, \Cal E)$.
Note that $\Cal E^*$ witnesses local trace definability of $\Sa M^* \models T$ in $\Sa N^* \models T^*$.
By saturation there is an elementary embedding $e$ of $\Sa O$ into $\Sa M^*$.
Finally, the collection of functions of the form $f \circ e$ witnesses local trace definability of $\Sa O$ in $\Sa N^*$.

\medskip
(3): Suppose that $\fik$ witnesses local trace definability of $\Sa M_1$ in $\Sa M_2$ and that $(g_i)_{i < \eta}$ witnesses local trace definability of $\Sa M_2$ in $\Sa M_3$.
It is easy to see that $(g_i \circ f_j)_{i < \eta, j < \kappa}$ witnesses local trace definability of $\Sa M_1$ in $\Sa M_3$, and this is a collection of $\kappa\eta$ functions.
\end{proof}

We consider structures admitting quantifier elimination in relational languages.
In the introduction we claimed that a theory is unstable if and only if it trace defines $(\Q; <)$.
This follows from (1) below and quantifier elimination for $(\Q; <)$.

\begin{fact}\label{fact:rel qe}
Suppose that $L$ is a relational language, $\Sa M$ is an $L$-structure with quantifier elimination, and $\Sa N$ is an arbitrary structure.
\begin{enumerate}[leftmargin=*]
\item $\Sa N$ trace defines $\Sa M$ if and only if there is an injection $\uptau \colon M \hookrightarrow N^n$ for some $n$ such that for any $k$-ary $R \in L$ there is $\Sa N$-definable $Y \subseteq N^{kn}$ satisfying
\[
\Sa M \models R(a_1, \ldots, a_k) \quad \Longleftrightarrow \quad (\uptau(a_1), \ldots, \uptau(a_k)) \in Y \quad \text{for all $a_1, \ldots, a_k \in M$.}
\]
\item $\Sa N$ trace defines $\Sa M$ if and only if $\Sa M$ embeds into an $\Sa N$-definable $L$-structure.
\item $\Sa N$ locally trace defines $\Sa M$ if and only if  $(M; R)$ embeds into an $\Sa N$-definable structure for any $R \in L$.
\end{enumerate}
\end{fact}

\begin{proof}
The left to right implication of (1) is immediate from Definition~\ref{def:trace}(2).
The other implication follows easily by quantifier elimination for $\Sa M$.
The right to left implication of (2) follows by applying transitivity of trace definability together with Facts~\ref{fact:trace embedd} and \ref{fact:interpret}.
We prove the other implication.
Suppose that $\uptau$ is as in (1).
For each $k$-ary $R \in L$ fix $\Sa N$-definable $Y_R \subseteq N^{kn}$ as in (1) and let $\Sa P$ be the $L$-structure with domain $N^n$ where each $R \in L$ is interpreted as $Y_R$.
Then $\Sa P$ is definable in $\Sa N$ and $\uptau$ gives an embedding $\Sa M \hookrightarrow \Sa P$.

\medskip
We prove (3).
The left to right implication is immediate from Definition~\ref{def:loc trace}(5).
We prove the other implication by applying Definition~\ref{def:loc trace}(4).
Suppose that $L^*$ is a finite relational language and $\Sa P$ is an $\Sa M$-definable $L^*$-structure.
Then there are distinct $R_1, \ldots, R_d \in L$ such that $\Sa P$ is quantifier-free definable in $(M; R_1, \ldots, R_d)$.
Note that any embedding of $(M; R_1, \ldots, R_d)$ into an $\Sa N$-definable structure induces an embedding of $\Sa P$ into another $\Sa N$-definable structure.
We show that $(M; R_1, \ldots, R_d)$ embeds into an $\Sa N$-definable structure.
By assumption there is an embedding $\uptau_i$ of $(M; R_i)$ into an $\Sa N$-definable structure $(Y_i; S_i)$ for each $i = 1, \ldots, d$.
Let $k_i$ be the arity of $R_i$ and let $\uppi_i$ be the projection $Y_\times := Y_1 \times \cdots \times Y_d \to Y_i$ for each $i$.
Let $\Sa Q$ be the structure with domain $Y_\times$ and a $k_i$-ary relation $S^*_i$ for each $i = 1, \ldots, d$ given by declaring $S^*_i(a_1, \ldots, a_{k_i})$ when $S_i(\uppi_i(a_1), \ldots, \uppi_i(a_{k_i}))$.
Then $\Sa Q$ is definable in $\Sa N$ and $(\uptau_1, \ldots, \uptau_d)$ gives an embedding $\Sa P \hookrightarrow \Sa Q$.
\end{proof}

Given a family $(L_i)_{i \in I}$ of languages let $L_\sqcup$ be the disjoint union of the $L_i$, considered as a $|I|$-sorted language in the natural way.
Given a family $(\Sa M_i)_{i\in I}$ of one-sorted structures we let $\bigsqcup_{i\in I} \Sa M_i$ be the disjoint union of the $\Sa M_i$ considered as an $L_\sqcup$-structure in the natural way.
If $|I| < \aleph_0$ then $\bigsqcup_{i \in I} \Sa M_i$ is finitely-sorted and therefore can be identified with a one-sorted structure in the usual way.
We need to consider infinite disjoint unions in regards to local trace definability at one point below, so we need to consider local trace definability between infinitely-sorted structures for this purpose.
We define local trace definability between infinitely-sorted structures with Definition~\ref{def:loc trace}(4) or by generalizing (3) from the same definition in an obvious way.
See \cite[\S~2.2]{trace2} for details.

\begin{fact}\label{fact:disjoint union}
Let $(\Sa M_i)_{i \in I}$ be a family of structures.
A theory $T$ (locally) trace defines $\bigsqcup_{i \in I} \Sa M_i$ if and only if $T$ (locally) trace defines each $\Sa M_i$.
\end{fact}

Fact~\ref{fact:disjoint union} is \cite[Lemma~2.21]{trace2}.
The left to right implication is clear.
The other implication follows from the fact that if $i_1, \ldots, i_k \in I$ are distinct then any subset of $M^{m_1}_{i_1} \times \cdots \times M^{m_k}_{i_k}$ definable in $\bigsqcup_{i \in I} \Sa M_i$ is a boolean combination of sets of the form $X_1 \times \cdots \times X_k$ where each $X_d \subseteq M^{m_d}_{i_d}$ is $\Sa M_d$-definable.

\medskip
We showed that many classification-theoretic properties are preserved under (local) trace definability in \cite{trace1, trace2}.
We only recall the few preservation results that we need at present.

\begin{fact}\label{fact:preserve}
Stability and $k$-$\nip$ are preserved under local trace definability for each $k \ge 1$.
Total transcendence and strong dependence are preserved under trace definability.
\end{fact}

Preservation of stability and $k$-$\nip$ is immediate from the characterizations of these properties stated in the introduction.
However, it also follows immediately from the usual definitions of instability and $k$-$\ip$.
Preservation of strong dependence also follows easily from the definition in terms of ict-patterns, see \cite[\S~4]{Simon-Book}.
Preservation of total transcendence can be proven using a type-counting argument.
One can also show that a theory is not totally transcendental if and only if it trace defines the structure with domain the Cantor set and unary relations defining all clopen sets.

\medskip
We say that a structure or theory is $\infty$-$\nip$ if it is $k$-$\nip$ for some $k \ge 1$ and is $\infty$-$\ip$ if it is not $\infty$-$\nip$.
Note that $\infty$-$\nip$ is preserved under local trace definability.

\begin{fact}\label{fact:loc max}
A theory is $\infty$-$\ip$ if and only if it locally trace defines every structure.
\end{fact}

Fact~\ref{fact:loc max} is \cite[Prop.~2.15]{trace2}.
The right to left direction follows as $\infty$-$\nip$ is preserved under local trace definability.
We sketch the other direction.
By Definition~\ref{def:loc trace}(5) $T$ locally trace defines every structure if and only if any structure in a language containing a single relation embeds into a structure definable in a model of $T$.
Furthermore, it follows easily from the definition of $k$-$\ip$ that if $T$ is $k$-$\ip$ then any $(k + 1)$-ary relation embeds into a relation definable in a model of $T$.

\medskip
We now recall the Shelah completion of a $\nip$ structure.
This is more commonly referred to as the ``Shelah expansion".
Let $\Sa M$ be a $\nip$ structure and $\Sa N$ be an $|M|^+$-saturated elementary extension of $\Sa M$.
The {\bf Shelah completion} $\Sh M$ of $\Sa M$ is the structure induced on $\Sa M$ by $\Sa N$.
It is easy to see that this structure does not depend on $\Sa N$ up to interdefinability.
Note that if $\Sa M$ expands a linear order then any convex subset of $M$ is definable in $\Sh M$.
Shelah showed that every $\Sh M$-definable subset of every $M^m$ is of the form $Y \cap M^m$ for $\Sa N$-definable $Y \subseteq N^m$~\cite{Shelah-external}.
It follows that $\Sh M$ is trace definable in $\Sa N$.
Furthermore, $\Sa M$ is clearly a reduct of $\Sh M$, so we have the following.

\begin{fact}\label{fact:she}
Any $\nip$ structure is trace equivalent to its Shelah completion.
\end{fact}

Finally, we will use two special cases of Fact~\ref{fact:jera}.
See \cite{jera} for a proof.

\begin{fact}\label{fact:jera}
Let $L$ be an arbitrary language.
Then the empty $L$-theory has a model completion which is complete when $L$ does not contain constants.
\end{fact}

\section{The theory of infinite-dimensional space over a model of $T$}\label{section:Dkap}
Fix a theory $T$ and an infinite cardinal $\kappa$.
We let $\dkap$ be the theory of two-sorted structures of the form $(M^\kappa, \Sa M, \Pi)$ where $\Sa M \models T$ and $\Pi$ is the collection of all coordinate projections $M^\kappa \to M$.
Note that $\dkap$ is definitionally equivalent to $D^\kappa(T^*)$ when $T$ is definitionally equivalent to $T^*$.
Now $\dkap$ is a two-sorted theory, but we will identify it with a one-sorted theory in the usual way when convenient.

\begin{proposition}\label{prop:pik1}
Let $T$ be an $L$-theory and $\kappa$ be infinite.
\begin{enumerate}[leftmargin=*]
\item $\dkap$ is complete.
\item $\dkap$ is axiomatized by axioms asserting that $(P, \Sa M, \Pi)$ is a model of $\dkap$ if and only if $\Sa M \models T$, $\Pi$ is a collection of $\kappa$ functions $P \to M$, and for any $a_1, \ldots, a_n \in M$, distinct $\uppi_1, \ldots, \uppi_n \in \Pi$, and $m \ge 1$, there are $m$ distinct elements $p$ of $P$ such that $\uppi_{i}(p) = a_i$ for each $i = 1, \ldots, n$.
\item If $\Sa M \models T$, $P$ is a set, and $\fik$ is a collection of functions $P \to M$, then there is a set $Q$ extending $P$ and a collection $\gik$ of functions $Q \to M$ such that each $g_i$ extends $f_i$ and $(Q, \Sa M, \gik) \models \dkap$.
\item Let $x_1, \ldots, x_n$ be variables of the first sort and $y_1, \ldots, y_m$ be variables of the second sort.
Then any formula in the $x_i$ and $y_i$ is equivalent in $\dkap$ to a boolean combination of formulas in the $x_i$ in the language of equality and formulas of the form 
\[
\vartheta(\uppi_1(x_{i_1}), \ldots, \uppi_d(x_{i_d}), y_1, \ldots, y_m)
\]
for some $L$-formula $\vartheta(z_1, \ldots, z_{d + m})$ and $\uppi_1, \ldots, \uppi_d \in \Pi$.
\end{enumerate}
Suppose furthermore that $T$ admits quantifier elimination.
Then
\begin{enumerate}[leftmargin=*]\setcounter{enumi}{4}
\item $\dkap$ admits quantifier elimination.
\item $\dkap$ is the model completion of the theory $\dkap_{\_}$ of two-sorted structures of the form $(P, \Sa M, \fik)$ where $P$ is a set, $\Sa M \models T$, and each $f_i$ is a function $P \to M$.
\end{enumerate}
\end{proposition}

\begin{proof}
The case of (4) when $T$ admits quantifier elimination follows directly from (5).
Hence the general case of (4) follows from (5) by Morleyization.
Also using Morleyiziation, it is sufficient to prove (1) and (2) under the assumption that $T$ admits quantifier elimination.
Furthermore, it is clear that $\dkap$ satisfies the axioms given in (2).
We therefore suppose that $T$ admits quantifier elimination, let $T^*$ be the theory given by the axioms in (2), and show that $T^*$ is complete, admits quantifier elimination, and is the model completion of $\dkap_{\_}$.
We also show that (3) holds with $\dkap$ replaced by $T^*$.
Completeness implies $T^* = \dkap$, so the proposition follows.

\medskip
First let $\Sa M$, $P$, and $\fik$ be as in (3).
Let $Q$ be the disjoint union of $P$ with $M^\kappa$.
For each $i < \kappa$ let $g_i$ be the function $Q \to M$ which agrees with $f_i$ on $P$ and agrees with the $i$th coordinate projection $M^\kappa \to M$ on $M^\kappa$.
Note that $(Q, \Sa M, \gik)$ satisfies $T^*$.

\medskip
We now suppose that $T$ admits quantifier elimination and show that $T^*$ is the model companion of $\dkap_{\_}$.
By the previous paragraph any model of $\dkap_{\_}$ embeds into a model of $T^*$.
We fix a model $(P, \Sa M, \fik)$ of $T^*$ and show that $(P, \Sa M, \fik)$ is existentially closed in the class of models of $\dkap_{\_}$.
Let $(P^*,  \Sa M^*, \fik)$ be a model of $\dkmin$ extending $(P, \Sa M, \fik)$.
Let $(Q, \Sa N, \gik)$ be a $\max(\kappa, |P^*|, |M^*|)^+$-saturated elementary extension of $(P, \Sa M, \fik)$.
It suffices to show that there is an embedding of $(P^*,  \Sa M^*, \fik)$ into $(Q, \Sa N, \gik)$ which fixes every element of both $P$ and $M$.
Our assumption of quantifier elimination for $T$ ensures that $\Sa M^*$ is an elementary extension of $\Sa M$.
Hence by saturation we may suppose that $\Sa M^*$ is an elementary substructure of $\Sa N$.
Applying saturation, we see that for every $p \in P^* \setminus P$ there are at least $|P^*|^+$ elements $q$ of $Q$ such that $g_i(q) = f_i(p)$ for every $i < \kappa$.
Hence the inclusion $P \to Q$ extends to an injection $\eta \colon P^* \to Q$ such that $g_i(\eta(p)) = f_i(p)$ for all $p \in P^*$ and $i < \kappa$.
Now $\eta$ and the inclusion $\Sa M^* \to \Sa N$ together give the required embedding of $(P^*, \Sa M^*, \fik)$ into $(Q, \Sa N, \gik)$.

\medskip
A model of $T^*_\forall$ is a structure $(P, \Sa M, \fik)$ where $\Sa M \models T_\forall$ and each $f_i$ is a function $P \to M$.
As $T$ admits quantifier elimination $T_\forall$ has the amalgamation property.
It easily follows that $T^*_\forall$ has the amalgamation property.
As $T^*$ is model complete it follows that $T^*$ has quantifier elimination~\cite[Thm.~8.4.1]{Hodges}.
It remains to show that $T^*$ is complete.
By model completeness of $T^*$ it suffices to show that any two models of $T^*$ jointly embed into a third.
This follows from another obvious amalgamation argument which is again left to the reader.
\end{proof}



\begin{lemma}\label{lem:c}
Let $T$ be a theory and $(P, \Sa M, \fik)$ be a model of $\dkap$ such that $|P| \le |M|$.
Let $\iota$ be an arbitrary injection $P \hookrightarrow M$.
Then $\fik$, $\iota$, and the identity $M \to M$ together witness local trace definability of $(P, \Sa M, \fik)$ in $\Sa M$.
It follows that $\dkap$ is locally trace equivalent to $T$.
\end{lemma}

Lemma~\ref{lem:c} is immediate from Proposition~\ref{prop:pik1}(4).
(Note that $\iota$ handles formulas in the language of equality with variables ranging over $P$.)
Suppose $(P, \Sa M, \fik) \models \dkap$.
Let $\Sa P$ be the structure induced on $P$ by $(P, \Sa M, \fik)$.
Then $\Sa P$ interprets $(P, \Sa M, \fik)$ as every $f_i$ is a surjection $P \to M$.
Hence an arbitrary structure (locally) trace defines $(P, \Sa M, \fik)$ if and only if it (locally) trace defines $\Sa P$.
We use this without mention below.

\begin{proposition}\label{prop:prekey}
Suppose that $\Sa O$ is a structure in a language of cardinality at most $\kappa$ and $\Sa M$ is an arbitrary structure.
\begin{enumerate}[leftmargin=*]
\item If $\Sa O$ is locally trace definable in $\Sa M$ then $\Sa O$ is trace definable in $(M^\kappa, \Sa M, \pik)$.
\item Suppose $|M|^\kappa = |M|$.
Then $\Sa O$ is locally trace definable in $\Sa M$ when $\Sa O$ is trace definable in $(M^\kappa, \Sa M, \pik)$.
\end{enumerate}
\end{proposition}

Of course, there are arbitrarily large cardinals $\lambda$ such that $\lambda^\kappa = \lambda$, so any structure has an elementary extension satisfying the condition on $|M|$ in (2).

\begin{proof}
Lemma~\ref{lem:c} shows that $(M^\kappa, \Sa M, \pik)$ is locally trace definable in $\Sa M$ when we have $|M|^\kappa = |M|$.
Hence (2) follows by transitivity of local trace definability.
Suppose $\Sa M$ locally trace defines $\Sa O$.
By Lemma~\ref{lem:ka}(1) this is witnessed by a collection $\fik$ of $\kappa$ functions $O \to M$.
Let $\uptau \colon O \to M^\kappa$ be the map taking each $a \in M$ to the tuple with $i$th coordinate $f_i(a)$.
We show that $\uptau$ witnesses trace definability of $\Sa O$ in $(M^\kappa, \Sa M, \pik)$.
Let $X \subseteq O^n$ be $\Sa O$-definable.
After possibly permuting the $f_i$ we have
\[
X = \{ (a_1, \ldots, a_n) \in O^n : (f_1(a_{i_1}), \ldots , f_m(a_{i_m})) \in Y \}
\]
for some $i_1, \ldots, i_m \in \nset$ and $\Sa M$-definable $Y \subseteq M^m$.
Now let $Z$ be the set of $(b_1, \ldots, b_n) \in (M^\kappa)^n$ with $(\uppi_1(b_{i_1}), \ldots , \uppi_m(b_{i_m})) \in Y$.
Then $Z$ is definable in $(M^\kappa, \Sa M, \pik)$ and $X$ is the pullback of $Z$ via $\uptau$.
\end{proof}

Proposition~\ref{prop:key} is our key result on $\dkap$.

\begin{proposition}\label{prop:key}
Let $T, T^*$ be theories and $\kappa \ge \aleph_0$.
\begin{enumerate}[leftmargin=*]
\item $T$ locally trace defines $T^*$, and this is witnessed by  $\le \kappa$ functions, if and only if $\dkap$ trace defines $T^*$.
\item If $|T^*| \le \kappa$ then $T$ locally trace defines $T^*$ if and only if $\dkap$  trace defines $T^*$.
\item If $|T|, |T^*| \le \kappa$ then $T$ locally trace defines $T^*$ if and only if $\dkap$ trace defines $D^\kappa(T^*)$.
\item If $|T|, |T^*| \le \kappa$ then $T$ is locally trace equivalent to $T^*$ if and only if $\dkap$ is trace equivalent to $D^\kappa(T^*)$.
\end{enumerate}
Furthermore we have the following for any theory $T$.
\begin{enumerate}[label=(\alph*),leftmargin=*]
\item  $\dkap$ is the unique theory modulo trace equivalence satisfying (1) for all theories $T^*$.
\item If $\kappa\ge |T|$ then $\dkap$ is the unique theory of cardinality $\le\kappa$ modulo trace equivalence satisfying (2) for all theories $T^*$ of cardinality $\le\kappa$.
\end{enumerate}
\end{proposition}

\begin{proof}
Note that (1) describes the class of theories that are trace definable in $\dkap$ and hence characterizes $\dkap$ up to trace equivalence.
Hence (a) follows from (1).
Likewise (b) follows from (2).
We prove (1) - (4).

\medskip
Lemma~\ref{lem:c} shows that $\dkap$ is locally trace definable in $T$ and that this is witnessed by $\kappa$ functions.
Hence if $T^*$ is trace definable in $\dkap$ then $T^*$ is locally trace definable in $T$ via $\kappa$ functions by Lemma~\ref{lem:ka}.
The right to left implication of (1) and (2) follows.

\medskip
The left to right implication of (1) follows by the proof of Proposition~\ref{prop:prekey}.
The left to right implication of (2) follows from the left to right implication of (1) and Lemma~\ref{lem:ka}.
Finally, (4) follows from (3) and (3) follows from (2) and local trace equivalence of $D^\kappa(T)$ with $T$.
\end{proof}

\begin{corollary}
\label{cor:gg}
Let $T$ be a theory and $\kappa=|T|$.
Then the following are equivalent.
\begin{enumerate}
\item $T$ is trace equivalent to $D^\kappa(T)$.
\item $T$ is trace equivalent to $D^\kappa(T^*)$ for some theory $T^*$ of cardinality $\le\kappa$.
\item Any theory of cardinality $\le\kappa$ which is locally trace definable in $T$ is already trace definable in $T$.
\end{enumerate}
\end{corollary}

\begin{proof}
Clearly (1) implies (2).
Equivalence of (1) and (3) follows by Proposition~\ref{prop:key}(b).
Suppose (2).
Then $T$ is locally trace equivalent to $T^*$ as $D^\kappa(T^*)$ is locally trace equivalent to $T^*$.
Hence $D^\kappa(T)$ is trace equivalent to $D^\kappa(T^*)$ by Proposition~\ref{prop:key}(4), and (1) follows.
\end{proof}

\begin{proposition}\label{prop:combine}
Let $T$ be a theory and $\kappa$ and $\lambda$ be infinite cardinals.
Then $D^\lambda(D^\kappa(T))$ is trace equivalent to $D^{\lambda + \kappa}(T)$.
\end{proposition}

\begin{proof}
Let $\eta = \lambda + \kappa$.
Local trace definability of $D^\kappa(T)$ in $T$ is witnessed by $\kappa$ functions and local trace definability of $D^\lambda(D^\kappa(T))$ in $D^\kappa(T)$ is witnessed by $\lambda$ functions.
Hence by Lemma~\ref{lem:ka}(3) local trace definability of $D^\lambda(D^\kappa(T))$ in $T$ is witnessed by $\eta$ functions.
Hence $D^\lambda(D^\kappa(T))$ is trace definable in $D^{\eta}(T)$ by Proposition~\ref{prop:key}(1).
Finally $D^\lambda(D^\kappa(T))$  interprets both $D^\kappa(T)$ and $D^\lambda(T)$ and $\eta$ is either $\kappa$ or $\lambda$, so $D^\lambda(D^\kappa(T))$  interprets $D^\eta(T)$.
\end{proof}

Below we give examples of families of theories $\Cal T$ such that every member of $\Cal T$ is trace equivalent to some $\dkap$ for a fixed $\infty$-$\nip$ theory $T$.
The following theorem largely  describes trace definability between elements of such $\Cal T$.

\begin{thm}\label{thm:distinct}
Let $T$ be an arbitrary theory and $\kappa, \eta \ge |T|$.
\begin{enumerate}[leftmargin=*]
\item If $T$ is not trace maximal and $\kappa > |T|$ then $T$ does not trace define $D^\kappa(T)$.
\item If $T$ is $\infty$-$\nip$ then $D^\kappa(T)$ trace defines $D^\eta(T)$ if and only if $\eta \le \kappa$.
\end{enumerate}
\end{thm}

It follows in particular that if $T$ is $\infty$-$\nip$ then the local trace equivalence class of $T$ contains a proper class of theories modulo trace equivalence.

\medskip
Theorem~\ref{thm:distinct}(1) is an application of Fact~\ref{fact:var case} below.
For any $m \ge 1$ and cardinal $\kappa \ge 1$ let $R^\kappa_m$  be the model completion of the theory of a set equipped with $\kappa$ relations, each of arity $m$.
This model completion exists and is complete by Fact~\ref{fact:jera}.
Set $R_m = R^1_m$ and note that $R_m$ is the theory of the \Fraisse limit of the class of finite $m$-ary relations.

\begin{fact}\label{fact:var case}
Suppose that $T$ is an arbitrary theory and $\kappa$ is a cardinal such that $\kappa > |T|$.
Fix $m \ge 1$.
If $T$ trace defines $R^\kappa_m$ then $T$ trace defines $R_{m + 1}$.
\end{fact}

Fact~\ref{fact:var case} follows from \cite[Prop.~2.5]{trace1}.
We sketch a proof.
Let $\Sa M = (M; (R_i)_{i < \kappa}) \models R^\kappa_m$ be $\aleph_1$-saturated and suppose that $\Sa N \models T$ trace defines $\Sa M$.
For the sake of simplicity we suppose that $M$ is a subset of $N$ and that $\Sa N$ trace defines $\Sa M$ via the inclusion $M \hookrightarrow N$.
For every $i < \kappa$ fix a parameter-free formula $\varphi_i(x_1,\ldots,x_{n_i},y_1,\ldots,y_m)$, and $\beta_i \in N^{n_i}$ such that
\[
\Sa N \models \varphi_i(\beta_i, a_1, \ldots, a_m) \quad \Longleftrightarrow \quad \Sa M \models R_i(a_1, \ldots
, a_m) \quad \text{for any $a_1, \ldots, a_m \in M$.}
\]
Now as $\kappa > |T| \ge \aleph_0$ there is a countably infinite $I \subseteq \kappa$ such that $i \mapsto n_i$ and $i \mapsto \varphi_i$ are constant on $I$.
Set $n = n_i$ and $\varphi = \varphi_i$ for any $i \in I$.
After possibly permuting suppose that $I = \omega$.
Let $\Sa O = (O; R)$ be the countable model of $R_m$ and let $(c_i)_{i < \omega}$ be an enumeration of $O$.
By saturation there is an injection $\uptau_1 \colon O \hookrightarrow M$ such that we have
\[
\Sa O \models R(c_i, a_1, \ldots, a_m) \quad \Longleftrightarrow \quad \Sa M \models R_i(\uptau_1(a_1), \ldots, \uptau_1(a_m)) \quad \text{for all $i < \omega$, $a_1, \ldots, a_m \in O$.}
\]
Let $\uptau_0$ be the map $O \to N^n$ given by $\uptau_0(c_i) = \beta_i$ for all $i < \omega$.
Consider $\uptau_1$ to be a map $O \to N$.
Now we have
\[
\Sa O \models R(a_0, \ldots, a_m) \quad \Longleftrightarrow \quad \Sa N \models \varphi(\uptau_0(a_0), \uptau_1(a_1), \ldots, \uptau_1(a_m)) \quad \text{for all $a_0, \ldots, a_m \in O.$}
\]

It follows by quantifier elimination for $\Sa O$ that $\Sa N$ trace defines $\Sa O$ via $(\uptau_0, \uptau_1)$.

\begin{proof}[Proof of Theorem~\ref{thm:distinct}]
(1):
We first prove the following.

\begin{Claim*}
$T$ is $\infty$-$\ip$ if and only if $D^{\aleph_0}(T)$ is trace maximal.
\end{Claim*}

\begin{claimproof}
By Fact~\ref{fact:loc max} a theory is $\infty$-$\ip$ if and only if it locally trace defines every structure.
The right to left implication follows as $T$ locally trace defines $D^{\aleph_0}(T)$.
For the other implication, note that if $T$ locally trace defines $(\Z; +, \times)$, then $D^{\aleph_0}(T)$ trace defines $(\Z; +, \times)$, hence $D^{\aleph_0}(T)$ is trace maximal as $(\Z; +, \times)$ is trace maximal.
\end{claimproof}

Suppose that $\kappa > |T|$ and that $T$ trace defines $D^\kappa(T)$.
We show that $T$ is trace maximal.
By the claim it suffices to show that $T$ locally trace defines every structure.
Let $L$ be a language containing a single $m$-ary relation for some $m \ge 1$.
By Definition~\ref{def:loc trace}(5) it is enough to show that any $L$-structure embeds into an $L$-structure definable in a model of $T$.
Now any $L$-structure embeds into a model of $R_m$, so it is enough to show that $T$ trace defines $R_m$.
We apply induction on $m$.
The case $m = 1$ is trivial.
Suppose that $T$ trace defines $R_m$.
Fact~\ref{fact:rel qe}(2) shows that $R_m$ locally trace defines $R^\kappa_m$.
It follows that $T$ locally trace defines $R_{m}^\kappa$, hence $T$ trace defines $R_{m}^\kappa$ as $|R^\kappa_{m}| = \kappa$.
Finally $T$ trace defines $R_{m + 1}$ by Fact~\ref{fact:var case}.

\medskip
(2):
Suppose that $T$ is $\infty$-$\nip$.
It suffices to suppose that $\kappa < \eta$ and show that $D^\kappa(T)$ does not trace define $D^\eta(T)$.
Now $D^\kappa(T)$ is $\infty$-$\nip$ as $\infty$-$\nip$ is preserved under local trace definability by Fact~\ref{fact:preserve}.
Hence $D^\kappa(T)$ does not trace define $D^\eta(D^\kappa(T))$ by (1).   
By Proposition~\ref{prop:combine} $D^\eta(D^\kappa(T))$ is trace equivalent to $D^\eta(T)$.
\end{proof}

\section{Criterion for trace equivalence to $\dkap$}\label{section:dkap}
We give some general results which will be used to show  that a theory $T^*$ is trace equivalent to $D^\kappa(T)$ for a simpler theory $T$.
In many of these cases $T^*$ is the relative model companion of the theory of a certain expansion of $T$ by unary functions.

\medskip
Let $L$ be a language and $L^*$ be a language extending $L$.
Let $T$ be a (possibly incomplete) $L$-theory and $T^*_0, T^*$ be (possibly incomplete) $L^*$-theories such that  $T^*_0 \subseteq T^*$ and $T$ is the $L$-reduct of both $T^*_0$ and $T^*$.
Then $T^*$ is the {\bf model companion of $T^*_0$ relative to $T$} when
\begin{enumerate}[leftmargin=*]
\item Any model of $T^*_0$ embeds into a model of $T^*$ in such a way that the induced embedding of $L$-structures is elementary.
\item If $\eta$ is an embedding between models of $T^*$ such that the induced embedding of $L$-structures is elementary then $\eta$ is elementary.
\end{enumerate}
For example $\dkap$ is the model companion of $\dkmin$ relative to the theory of two-sorted structures of the form $(P, \Sa M)$ where $P$ is a set and $\Sa M \models T$.
It is easy to see that the relative model companion is unique up to logical equivalence when it exists.

\medskip
Now suppose in addition that $L^*$ is an expansion of $L$ by functions.
Then we say that $T^*$ has {\bf quantifier elimination relative to $T$} if every $L^*$-formula in the variables $x_1, \ldots, x_n$ is equivalent in $T^*$ to a formula of the form 
\[
\vartheta(x_1, \ldots, x_n, t_1, \ldots, t_m)
\]
where $\vartheta(y_1, \ldots, y_{n + m})$ is an $L$-formula and $t_1, \ldots, t_m$ are $L^* \setminus L$-terms in the $x_i$.
Note that (2) above is satisfied when $T^*$ admits quantifier elimination relative to $T$.

\begin{fact}\label{fact:relqe}
Let $T$ and $T^*$ be as above and suppose $L^*$ is an expansion of $L$ by $\lambda$ unary functions.
If $T^*$ admits quantifier elimination relative to $T$ then $T^*$ is locally trace equivalent to $T$ and local trace definability of $T^*$ in $T$ is witnessed by $\lambda + \aleph_0$ functions.
\end{fact}

Fact~\ref{fact:relqe} is clear from the definitions.
We now prove a general lemma.

\begin{lemma}\label{lem;g}
Fix structures $\Sa O, \Sa M$ and $\kappa \ge \aleph_0$.
Suppose that $O\subseteq M^m$ and suppose that $\Sa M$ trace defines $\Sa O$ via the inclusion $O\to M^m$.
Let $X$ be an $\Sa M$-definable subset of $M^n$ and $\Cal E$ be a cardinality $\kappa$ family of $\Sa M$-definable functions $X \to M^m$ such that for any $b_1,\ldots,b_k \in O$ and distinct $g_1, \ldots, g_k \in \Cal E$ there is $p\in X$ satisfying $g_{i}(p )= b_i$ for each $i = 1,\ldots,k$.
Then $\Th(\Sa M)$ trace defines $D^\kappa(\Th(\Sa O))$.
\end{lemma}

\begin{proof}
Fix an enumeration $(g_i)_{i<\kappa}$ of $\Cal E$.
After possibly replacing $\Sa M$ with a $\max(\kappa,|O|)^+$-saturated elementary extension we suppose that for any sequence $(b_i)_{i<\kappa}$ of elements of $O$ there is $a_p \in X$ such that $g_i(a_p)=b_i$ for all $i<\kappa$.
After possibly Morleyizing suppose that $\Sa O$ admits quantifier elimination in a relational language $L$.
The proof of Fact~\ref{fact:rel qe}(2) shows that $\Sa O$ is a substructure of an $\Sa M$-definable $L$-structure $\Sa P$ with domain $M^m$.
Let $\uppi_i \colon O^\kappa \to O$ be the projection onto the $i$th coordinate for each $i < \kappa$.
Proposition~\ref{prop:pik1}(5) shows that $(O^\kappa, \Sa O, (\uppi_i)_{i<\kappa})$ admits quantifier elimination as $\Sa O$ admits quantifier elimination.
The inclusion $O \to M^m$ and the map $O^\kappa \to X$ given by $p \mapsto a_p$ gives an embedding of $(O^\kappa, \Sa O, (\uppi_i)_{i<\kappa})$ into $(X, \Sa P, (g_i)_{i<\kappa})$ and $(X, \Sa P, (g_i)_{i<\kappa})$ is definable in $\Sa M$.
Hence $\Sa M$ trace defines $(O^\kappa, \Sa O, (\uppi_i)_{i<\kappa})$ by Fact~\ref{fact:trace embedd}.
\end{proof}

We now consider structures that admit ``space-filling curves".

\begin{proposition}
\label{prop;ginsep}
Suppose $\Sa M$ admits either a definable injection $M^k \to M$ or a definable surjection $M\to M^k$ for some $k\ge 2$.
Then $\Th(\Sa M)$ is trace equivalent to $D^{\aleph_0}(\Th(\Sa M))$.
\end{proposition}

Recall that $g^{(n)}$ is the $n$-fold compositional iterate of a function $g \colon X \to X$ for each $n \ge 1$.

\begin{proof}
It suffices to show that $\Th(\Sa M)$ trace defines $D^{\aleph_0}(\Th(\Sa M))$.
Fix $p\in M^k$.
If $f \colon M^k \to M$ is a definable injection then we define a surjection $g \colon M \to M^k$ by declaring $g(a) = f^{-1}(a)$ when $a$ is in the image of $f$ and otherwise set $g(a) = p$.
Hence there is a definable surjection $g \colon M \to M^k$ for some $k \ge 2$.
Composing with any coordinate projection $M^k \to M^2$ reduces to the case $k = 2$.
Let $g_1, g_2$ be the definable functions $M \to M$ such that we have $g(a) = (g_1(a), g_2(a))$ for all $a \in M$.
Let $h_1 = g_1$ and $h_{n} = g_1 \circ g^{(n-1)}_2$ for all $n \ge 2$.
\begin{Claim*}
For any  $b_1,\ldots,b_n \in M$ there is $\gamma \in M$ such that $h_i(\gamma) = b_i$ for $i = 1,\ldots,n$.
\end{Claim*}
By Lemma~\ref{lem;g} it is enough to prove the claim.
Fix $c \in M$ and let $s \colon M^2 \to M$ be a section of $f$, so $s(a_1,a_2) = b$ implies $a_i = g_i(b)$ for $i = 1,2$. Let
$$\gamma = s(b_1, s(b_2,\ldots,s(b_{n-1}, s(b_n,c))\ldots))$$
Then $h_i(\gamma) = b_i$ for each $i = 1, \ldots, n$.
\end{proof}

We now give three general lemmas that we use to show that various examples of an expansion of a theory $T$ by unary functions are trace equivalent to $D^\kappa(T)$.
Below ``independence" means independence with respect to model-theoretic algebraic closure.

\begin{lemma}\label{lem:dkap}
Suppose that $T$ is an $L$-theory.
Let $L^*$ be an expansion of $L$ by a collection of $\kappa \ge \aleph_0$ unary functions.
Suppose that $T^*$ is an $L^*$-theory extending $T$ and $T^*_0$ is a (possibly incomplete) $L^*$-theory contained in $T$ such that we have the following.
\begin{enumerate}[leftmargin=*]
\item $T^*$ is the model companion of $T^*_0$ relative to $T$ and $T^*$ admits quantifier elimination relative to $T$.
\item If $A \subseteq\Sa M \models T$ and $(f_i)_{i < \kappa}$ is a family of functions $A \to M$  so that $A$ is independent over $\bigcup_{i < \kappa} f_i(A)$, then each $f_i$ extends to a  function $f^*_i \colon M \to M$ in such a way that $(\Sa M, (f^*_i)_{i < \kappa}) \models T^*_0$.
\end{enumerate}
Then $T^*$ is trace equivalent to $D^\kappa(T)$.
\end{lemma}

\begin{proof}
It follows from Fact~\ref{fact:relqe} and (1) that $T^*$ is locally trace definable in $T$, and this is witnessed by $\kappa$ functions.
Hence $T^*$ is trace definable in $D^\kappa(T)$ by Proposition~\ref{prop:key}(1).
Fix $(P, \Sa M, \fik) \models D^{\kappa}(T)$.
Let $\Sa N$ be a $\max(|P|, \kappa)^+$-saturated elementary extension of $\Sa M$.
Up to isomorphism, we may suppose that $P$ is a subset of $N$ which is independent over $\Sa M$.
Applying (2), we get a function $f^*_i \colon N \to N$ extending $f_i$ for each $i < \kappa$ such that $(\Sa N, (f^*_i)_{i < \kappa}) \models T^*_0$.
Let $(\Sa N^*, (f^*_i)_{i < \kappa})$ be a model of $T^*$ extending $(\Sa N, (f^*_i)_{i < \kappa})$ such that $\Sa N$ is an elementary submodel of $\Sa N^*$.
An application of Lemma~\ref{lem;g} shows that $(\Sa N^*, (f^*_i)_{i < \kappa})$ trace defines a model of $D^\kappa(T)$.
\end{proof}

\begin{lemma}\label{lem:two un}
Suppose that $T$ is an $L$-theory, and $L^*$ is an expansion of $L$ by at least two and at most $\aleph_0$ unary functions.
Suppose that $T^*$ is an $L^*$-theory extending $T$ and $T^*_0$ is a (possibly incomplete) $L^*$-theory contained in $T^*$ such that we have the following.
\begin{enumerate}[leftmargin=*]
\item $T^*$ is the model companion of $T^*_0$ relative to $T$ and $T^*$ admits quantifier elimination relative to $T$.
\item There are $f, g \in L^* \setminus L$ such that if $\Sa M \models T^*_0$, $\Sa N$ is an elementary extension of the $L$-reduct of $\Sa M$, $b, b^* \in M$, and $a \in N \setminus M$, then there is a model $\Sa N^*$ of $T^*_0$ expanding $\Sa N$ which satisfies $f(a) = b$, $g(a) = b^*$.
\end{enumerate}
Then $T^*$ is trace equivalent to $D^{\aleph_0}(T)$.
\end{lemma}

\begin{proof}
After possibly Morleyizing we suppose $T$ admits quantifier elimination.
Then $T^*$ is the model companion of $T^*_0$.
By Fact~\ref{fact:relqe} $T^*$ is locally trace definable in $T$, and this is witnessed by countably many functions.
Hence $T^*$ is trace definable in $D^{\aleph_0}(T)$ by Proposition~\ref{prop:key}(1).
We show that $T^*$ trace defines $D^{\aleph_0}(T)$.
Fix $\Sa M \models T^*$ and let $f,g$ be as in (2).
By Proposition~\ref{prop;ginsep} it is enough to show that the map $M \to M^2$ given by $a \mapsto (f(a), g(a))$ is surjective.
We fix $b, b^* \in M$ and produce $a \in M$ such that $f(a) = b$ and $g(a) = b^*$.
By (2) there is a model $\Sa N^*$ of $T^*_0$ such that the $L$-reduct of $\Sa N^*$ extends $\Sa M$ and we have $f(a) = b$ and $g(a) = b^*$ for some $a \in N \setminus M$.
Now $\Sa M$ is existentially closed in $\Sa N^*$, so there is such an $a$ in $M$.
\end{proof}

A {\bf definable topology} on an $L$-theory $T$ is an $L$-formula $\varphi(x, y_1, \ldots, y_n)$  such that the collection of sets of the form $\{ \alpha \in M : \Sa M \models \varphi(\alpha, \beta)\}$ for $\beta \in M^{n}$ is a basis for a topology on $M$ for any $\Sa M \models T$.
In this situation we also equip each $M^m$ with the product topology.

\begin{lemma}\label{lem:one un}
Suppose that $T$ is a $\nip$ $L$-theory equipped with a definable topology.
Let $L^*$ be an expansion of $L$ by $\kappa \ge 1$ unary functions.
Suppose that $T^*$ is an $L^*$-theory extending $T$ and satisfying the following.
\begin{enumerate}[leftmargin=*]
\item $T^*$ has quantifier elimination relative to $T$.
\item If $(\Sa M, \fik) \models T^*$ then $\{(f_{i_1}(a), \ldots, f_{i_n}(a)) : a \in M\}$ is dense in $M^{n}$ for any distinct $i_1, \ldots, i_n < \kappa$ and $\{(a, f_i(a), f^{(2)}_i(a),\ldots,f^{(m)}_i(a)) : a \in M\}$ is  dense in $M^{m + 1}$ for each $i < \kappa$ and $m \ge 1$.
\item There is $(\Sa N, \fik) \models T^*$, an $\Sh N$-definable equivalence relation $E$ on $N$, and an $\Sh N$-definable open subset $V$ of $N$ such that every $E$-class is open and the structure induced on $V/E$ by $\Sh N$ is an expansion of a model $\Sa O$ of $T$ up to interdefinability.
\end{enumerate}
Then $T^*$ is trace equivalent to $D^{\kappa + \aleph_0}(T)$.
\end{lemma}

Recall that $\Sh M$ is the Shelah completion of a $\nip$ structure $\Sa M$.

\begin{proof}
As above (1) shows that $T^*$ is trace definable in $D^{\kappa + \aleph_0}(T)$.
It follows that $T^*$ is $\nip$ by Fact~\ref{fact:preserve}.
Let $(\Sa N, \fik), V, E$ be as in (3).
Note that $(\Sh N, \fik)$ is a reduct of the Shelah completion of $(\Sa N, \fik)$.
It follows by Fact~\ref{fact:she} that $(\Sh N, \fik)$ is trace definable in $T^*$.
Hence it is enough to show that $\Th(\Sh N, \fik)$ trace defines $D^{\kappa + \aleph_0}(T)$.
Let $\uppi \colon N \to N/E$ be the quotient map, let $O = V/E$, and let $\Sa O \models T$ be as in (3).

\medskip
We first treat the case when $\kappa$ is finite.
Let $f = f_0$ and let $g_i \colon N \to N/E$ be given by $g_i(a) = \uppi(f^{(i)}(a))$ for all $i \in \N$.
By Lemma~\ref{lem;g} it is enough to show that for any $b_0,\ldots,b_n \in O$ there is $a \in N$ such that $g_i(a) = b_i$ for each $i = 0,  \ldots, n$.
So we need $a \in N$ such that each $f^{(i)}(a)$ is in $\uppi^{-1}(b_i)$.
This follows from (2) as each $\uppi^{-1}(b_i)$ is open and nonempty.

\medskip
We now suppose $\kappa$ is infinite.
For each $i < \kappa$ let $g_i \colon N \to N/E$ be given by $g_i(a) = \uppi(f_i(a))$.
It is enough to show that for any $b_1, \ldots, b_n \in O$ and distinct $i_1, \ldots, i_n < \kappa$ there is $a \in N$ such that $g_{i_j}(a) = b_j$ for $j = 1, \ldots, n$.
Again, this follows from (2).
\end{proof}

\section{The trivial theory}
We consider the trivial theory $\triv$ of an infinite set equipped with equality.
We show that several theories are trace equivalent to $D^\kappa(\triv)$.
We first recall the relevant theories.
We let $E_\kappa$ be the model completion of the theory of a set equipped with $\kappa$ equivalence relations.
It is well-known that $E_\kappa$ exists.
Let $F_\kappa$ be the model completion of the theory of a set equipped with $\kappa$ unary functions for each $\kappa \ge 1$.
This theory exists by Fact~\ref{fact:jera}.
Let $A_\kappa$ be the model completion of the theory of a set $M$ equipped with $\kappa$ commuting functions $M\to M$.
Existence of $A_\kappa$ follows from Gould's results that the theory of actions of a coherent semigroup has a model completion \cite[Thm.~6]{Gould1987} and  that any free abelian semigroup is coherent \cite[Thm.~4.3]{Gould1992}.
Each theory admits quantifier elimination as they are the model completions of universal theories.
Finally, each of these theories is complete as it is easy to see that any two models jointly embed into a third.

\medskip
A \textbf{Jo\'nsson-Tarski algebra} is a structure $(M;p,l,r)$ where $p$ is a function $M^2\to M$ and $l,r$ are functions $M\to M$ such that for all $a,a^*\in M$ we have:
\begin{enumerate}[leftmargin=*]
\item $p(l(a),r(a))=a$
\item $l(p(a,a^*))=a$ and $r(p(a,a^*))=a'$.
\end{enumerate}
So the theory of J\'onsson-Tarski algebras is definitionally equivalent to the theory of a set $M$ equipped with a bijection $M \to M^2$.
J\'onsson and Tarski showed that the free J\'onsson-Tarski algebra on $n$ generators is isomorphic to the free J\'onsson-Tarski algebra on one generator for any $n \ge 2$~\cite{JoTa}.
We call this structure the \textbf{free J\'onsson-Tarski algebra}.
Bouscaren and Poizat showed that the theory of locally free (i.e. every finitely generated subalgebra is free) J\'onsson-Tarski algebras is complete and admits quantifier elimination~\cite{Bouscaren_Poizat_1988}.
Hence the theory of the free J\'onsson-Tarski algebra is the theory of locally free J\'onsson-Tarski algebras.

\begin{proposition}\label{prop;triv}
\hspace{.000001cm}
\begin{enumerate}[leftmargin=*]
\item If $\kappa \ge \aleph_0$ then both $A_\kappa$ and $E_\kappa$ are trace equivalent to $D^\kappa(\triv)$.
\item If $\kappa \ge 2$ then $F_\kappa$ is trace equivalent to $D^{\kappa + \aleph_0}(\triv)$.
\item The theory of the free J\'onsson-Tarski algebra is trace equivalent to $D^{\aleph_0}(\triv)$.
\end{enumerate}
\end{proposition}

The bounds on $\kappa$ are sharp as $A_\kappa$ and $E_\kappa$ are totally transcendental when $\kappa$ is finite, $D^{\aleph_0}(\triv)$ is not totally transcendental, and total transcendence is preserved under trace definability by Fact~\ref{fact:preserve}.
See Lemma~\ref{lem:A fin} below for total transcendence of $A_\kappa$.


\begin{proof}
We first consider $F_\kappa$.
The case when $\kappa$ is infinite follows easily by applying Lemma~\ref{lem:dkap} with $T = \triv$, $T^*_0$ the theory of a set equipped with $\kappa$ unary functions, and $T^* = F_\kappa$.
The case when $\kappa$ is finite follows by applying Lemma~\ref{lem:two un} with $T$, $T^*_0$, and $T^*$ as before.

\medskip
Now suppose that $\kappa$ is infinite and consider $A_\kappa$.
We apply Lemma~\ref{lem:dkap} with $T = \triv$, $T^*_0$ the theory of a set equipped with $\kappa$ commuting maps, and $T^* = A_\kappa$.
It is enough to fix a set $M$, a subset $A\subseteq M$, and a family $\fik$ of functions $A \to M$ such that $A$ is disjoint from $\bigcup_{i < \kappa} f_i(A)$, and show that each $f_i$ extends to a function $f^*_i \colon M \to M$ such that the $f^*_i$ commute.
Fix arbitrary $p \in M$ and declare $f^*_i(a) = f_i(a)$ when $a \in A$ and otherwise $f^*_i(a) = p$.
Then $f^*_i \circ f^*_j$ is the constant $p$ function for any $i,j < \kappa$, hence the $f^*_i$ commute.

\medskip
We now consider $E_\kappa$ for $\kappa \ge \aleph_0$.
Fix $(P, M ; \fik) \models D^\kappa(\triv)$.
For each $i < \kappa$ let $E_i$ be the equivalence relation on $P$ given by declaring $E_i(a, a^*)$ if and only if $f_i(a) = f_i(a^*)$.
It is easy to see that $(M ; (E_i)_{i < \kappa}) \models E_\kappa$, so $F_\kappa$ interprets $E_\kappa$.
Let  $\Sa P$ be the induced structure on $P$.
It suffices to show that $E_\kappa$ trace defines $\Sa P$.
Let $P^* = P \times \{1, 2\}$.
For each $i, j < \kappa$ let $f_{i, j}$ be the function $P^* \to P$ given by 
\[
f_{i, j}(a, m) =
\begin{cases}
f_i(a) & \text{when $m = 1$ } \\
f_j(a) & \text{when $m = 2$.}
\end{cases}
\]
Furthermore let $E_{i, j}$ be the equivalence relation on $P^*$ given by declaring $E_{i, j}(b, b^*)$ when $f_{i, j}(b) = f_{i, j}(b^*)$ for each $i, j < \kappa$.
Let $\Sa N = (N ; (E_{i, j})_{i, j < \kappa})$ be a model of $E_\kappa$ extending $(P^* ; (E_{i, j})_{i, j < \kappa})$.
Given $i = 1, 2$ let $\uptau_i \colon P \to N$ be given by $\uptau_i(a) = (a, i)$.
We show that $\uptau_1, \uptau_2$ witness trace definability of $\Sa P$ in $\Sa N$.
By quantifier elimination for $D^\kappa(\triv)$ it is enough to consider sets of the form
\[
X = \{ (a, b) \in P^2 : f_i(a) = f_j(b) \}
\]
for some $i, j < \kappa$.
In this case we have $(a, b) \in X$ if and only if $\Sa N \models E_{i , j}(\uptau_1(a), \uptau_2(b))$.

\medskip
Now let $\mathrm{JT}$ be the theory of locally free J\'onsson-Tarski algebras.
Proposition~\ref{prop;ginsep} shows that $\mathrm{JT}$ is trace equivalent to $D^{\aleph_0}(\mathrm{JT})$.
By Proposition~\ref{prop:key}(4) it is enough to show that $\mathrm{JT}$ is locally trace equivalent to $\triv$.
Fix $\Sa M=(M;p,r,l)\models\mathrm{JT}$.
Any term $t(x_1,\ldots,x_n)$  in $\Sa M$ is equivalent to a term of the form $t^*(s_1(x_{i_1}),\ldots,s_m(x_{i_m}))$ for some $\{p\}$-term $t^*(y_1,\ldots,y_m)$, $\{l,r\}$-terms $s_1,\ldots,s_m$, and $i_1,\ldots,i_m\in\nset$.
Now $t^*$ is built up by composing $p$, so it follows that there are $\{l,r\}$-terms $u_1,\ldots,u_m$ such that
\[
t^*(y_1,\ldots,y_m) = z\quad\Longleftrightarrow\quad \bigwedge_{j = 1}^{m} u_j(z) = y_j.
\]
Hence we have 
\[
t(x_1,\ldots,x_n) = z\quad\Longleftrightarrow\quad \bigwedge_{j=1}^{n} u_j(z) = s_j(x_{i_j}).
\]
It follows that $(M;r,l)$ is interdefinable with $\Sa M$ and admits quantifier elimination.
Hence $(M;r, l)$ admits quantifier elimination relative to the trivial structure on $M$, and so $(M; r, l)$ is locally trace definable in the trivial structure on $M$ by Fact~\ref{fact:relqe}.
\end{proof}

\section{Vector spaces, modules, and $\mathrm{DOAG}$}
Let $\F$ be a field and $\vvec_\F$ be the theory of infinite $\F$-vector spaces.
We show that certain theories of modules over $\F$-algebras are trace equivalent to $D^\kappa(\vvec_\F)$.
An $\F$-algebra is a unitary ring $R$ with an embedding of $\F$ into the center of $R$.
In particular an $\F_p$-algebra is just a characteristic $p$ ring and a $\Q$-algebra is just a ring containing $\Q$ as a subring.
Let $R$ be a ring.
We follow the usual convention by considering an $R$-module to be a first order structure consisting of an abelian group $V$ equipped with unary functions $(\lambda_r)_{r\in R}$ such that $r\mapsto\lambda_r$ gives a homomorphism from $R$ to the endomorphism ring of $V$.
We first recall a basic fact about theories of modules.

\begin{fact}\label{fact:R-mod complete}
Let $R$ be a ring.
If the theory of $R$-modules has a model companion then the model companion is a complete theory.
\end{fact}

\begin{proof}
It is enough to show that any two $R$-modules jointly embed into a third.
This follows by taking the direct sum.
\end{proof}

\begin{fact}\label{fact:vec}
Any two vector spaces over division rings of the same characteristic are locally trace equivalent.
\end{fact}

Fact~\ref{fact:vec} is \cite[Prop.~4.1]{trace2}\footnote{It follows by noting that if $\D$ is a division ring and $\Sa V$ is a $\D$-vector space then the collection of functions $v \mapsto \lambda v$, $\lambda \in \D$ witnesses local trace definability of $\Sa V$ in its underlying additive group.
}.
We say $\F$ is a {\bf prime field} if $\F$ is either $\Q$ or $\F_p$ for some prime $p$.
Let $\E$ be the prime subfield of $\F$.
By Fact~\ref{fact:vec} $\vvec_\F$ is trace equivalent to $\vvec_\E$.
Hence $D^\kappa(\vvec_\F)$ is trace equivalent to $D^\kappa(\vvec_\E)$ when $\kappa \ge |\F|$.

\medskip
Let $\F[x_i]_{i<\kappa}$ be the polynomial ring over $\F$ in $\kappa$ variables for any cardinal $\kappa\ge 1$.
An $\F[x_i]_{i<\kappa}$-module is just an $\F$-vector space equipped with $\kappa$ commuting endomorphisms.
We also let $\F\langle x_i\rangle_{i<\kappa}$ be the free $\F$-algebra on $\kappa$ generators, i.e. the polynomial ring in $\kappa$ non-commuting variables.
An $\F\langle x_i\rangle_{i < \kappa}$-module is just an $\F$-vector space equipped with $\kappa$ endomorphisms.
Eklof and Sabbagh showed that if $R$ is a coherent ring then the theory of $R$-modules has a model completion~\cite{EKLOF1971251}.
Now $\F[x_i]_{i<\kappa}$ is coherent by \cite[Thm.~1]{sabbagh-coherence}, hence the theory of $\F[x_i]_{i<\kappa}$-modules has a model completion which we denote by $\mathrm{CM}_{\kappa,\F}$.
Furthermore $\F\langle x_i\rangle_{i<\kappa}$ is also coherent~\cite[Cor.~2.2]{choo-lam-luft}, hence the theory of $\F\langle x_i\rangle_{i<\kappa}$-modules has a model completion $\mathrm{NCM}_{\kappa,\F}$.
Both $\mathrm{CM}_{\kappa,\F}$ and $\mathrm{NCM}_{\kappa,\F}$ admit quantifier elimination as they are model completions of universal theories and both theories are complete by Fact~\ref{fact:R-mod complete}.

\begin{proposition}\label{prop:dkaplin}
Let $\F$ be a field.
\begin{enumerate}[leftmargin=*]
\item If $\kappa \ge  \aleph_0$ then $\mathrm{CM}_{\kappa, \F}$ is trace equivalent to  $D^\kappa(\vvec_\F)$.
\item If $\kappa \ge 2$ then $\mathrm{NCM}_{\kappa, \F}$ is trace equivalent to $D^{\kappa + \aleph_0}(\vvec_\F)$.
\item If $\F$ is countable, $V$ is an infinite-dimensional $\F$-vector space, and $f$ is a linear surjection $V \to V^n$ for some $n \ge 2$ then the theory of $(V, f)$ is trace equivalent to $D^{\aleph_0}(\vvec_\F)$.
\end{enumerate}
\end{proposition}

\begin{proof}
The proof of (3) is given after Proposition~\ref{prop:F}.
We prove (1).
Suppose $\kappa \ge \aleph_0$.
Let $V$ be an $\F$-vector space.
By Lemma~\ref{lem:dkap} it is enough to suppose $A \subseteq V$  and $(f_i)_{i < \kappa}$ is a collection of functions $A \to V$ such that $A$ is linearly independent over $\bigcup_{i < \kappa} f_i(A)$ and show that each $f_i$ extends to an endomorphism $f^*_i$ of $V$ such that the $f^*_i$ commute.
Let $W$ be the vector subspace of $V$ spanned by $\bigcup_{i < \kappa} f_i(A)$.
Let $A^*$ be a maximal subset of $V$ which contains $A$ and is independent over $W$.
Then for each $i < \kappa$ there is an endomorphism $f^*_i$ of $V$ which agrees with $f_i$ on $A$, takes $A^*$ to $W$, and vanishes on $W$.
Then we have $f^*_i \circ f^*_j = 0$ for any $i, j < \kappa$, so the $f^*_i$ trivially commute.

\medskip
The case of (2) when $\kappa$ is infinite follows from the proof of (1).
The case when $2 \le \kappa < \aleph_0$ follows by applying Lemma~\ref{lem:two un} with $T = \vvec_\F$, $T^* = \mathrm{NCM}_{\kappa, \F}$, and $T^*_0$ the theory of an $\F$-vector space equipped with $\kappa$ endomorphisms.
\end{proof}

Our next goal is to prove some things about trace definability in modules.
We first give a general result about abelian structures.
Let $\Sa A$ be an expansion of an abelian group $A$.
Then $\Sa A$ is an \textbf{abelian structure} if $\Sa A$ is interdefinable with an expansion of $A$ by some collection of subgroups of various $A^n$.

\begin{fact}\label{fact:abelian structure}
Let $\Sa A$ be an expansion of an abelian group $A$.
Then the following are equivalent.
\begin{enumerate}[leftmargin=*]
\item $\Sa A$ is one-based.
\item $\Sa A$ is an abelian structure.
\item Any definable subset of any $A^n$ is a boolean combination of cosets of  definable subgroups of $A^n$.
\end{enumerate}
\end{fact}

\begin{proof}
Hrushovski and Pillay showed that (1) and (3) are equivalent~\cite{HP-weakly-normal}.
It is clear that (3) implies (2).
See \cite[Thm.~4.2.8]{Wagner_1997} for a proof that (2) implies (3).
\end{proof}

\begin{proposition}
\label{prop:abelai}
Let $\Sa A$ be an abelian structure.
Then $\Sa A$ is locally trace equivalent to the disjoint union  $\Sa D$  of all abelian groups $A^n/B$ for $n \ge 1$ and $B$ an $\Sa A$-definable subgroup of $A^n$.
\end{proposition}

\begin{proof}
It is clear that $\Sa A$ interprets $\Sa D$.
We show that $\Sa D$ locally trace defines $\Sa A$.
Given $n \ge 1$, a definable subgroup $B \subseteq A^n$, and $i \in \nset$ we let
\begin{enumerate}[leftmargin=*]
\item $\uppi_B$ be the quotient map $A^n \to A^n/B$,
\item $e_{n, i}$ be the map $A \to A^n$ taking $a \in A$ to the vector with $i$th coordinate $a$ and all other coordinates $0$,
\item and $\uptau_{B, i}$ be the map $A \to A^n/B$ given by composing $e_{n, i}$ and $\uppi_B$.
\end{enumerate}
Each $\uptau_{B, i}$ is a group morphism  and we have  $$\uppi_B(\alpha_1,\ldots,\alpha_n) = \uptau_{B, 1}(\alpha_1) + \cdots + \uptau_{B,n}(\alpha_n) \quad\text{for all} \quad \alpha_1,\ldots,\alpha_n\in A.$$
As $\Sa A$ is an abelian structure it follows that every $\Sa A$-definable subset of $A^n$ is a boolean combination of sets of the form
\[
\{(\alpha_1,\ldots,\alpha_n) \in A^n : \uptau_{B, 1}(\alpha_1) + \cdots + \uptau_{B,n}(\alpha_n) = \gamma \}
\]
for some definable subgroup $B \subseteq A^n$ and $\gamma \in A^n/B$.
Hence the collection of all $\uptau_{B, i}$ witnesses local trace definability of $\Sa A$ in $\Sa D$.
\end{proof}

We now give a lemma which will be used to understand definable subgroups in modules.

\begin{lemma}\label{lem:vect bool}
Let $V$ be a vector space over a characteristic zero field $\F$.
Then any subgroup of $V$ that is a boolean combination of affine subspaces of $V$ is a subspace.
\end{lemma}

Lemma~\ref{lem:vect bool} is probably not original.
Thanks to Will Johnson for helping with the proof.

\begin{proof}
We first prove the following.
\begin{Claim*}
Suppose that $W_1, \ldots, W_n$ are proper subspaces of $V$, $C_i$ is a coset of $W_i$ for each $i \in \nset$, and let $X = V \setminus (C_1 \cup \cdots \cup C_n)$.
Then $V = X - X$.
\end{Claim*}

\begin{claimproof}
Suppose otherwise and fix $v \in V \setminus (X - X)$.
Then $X \cap (X + v) = \varnothing$, and it follows that $C_1, \ldots, C_n, C_1 + v, \ldots, C_n + v$ cover $V$.
Hence $V$ is covered by finitely many cosets of the $W_i$.
By Neumann's lemma~\cite[Lemma~4.2.1]{Hodges} some $W_i$ has finite index in $V$.
Finally $V/W_i$ is an $\F$-vector space, so $V = W_i$ as $\F$ is characteristic zero, contradiction.
\end{claimproof}

Now suppose that $H$ is a subgroup of $V$ which is a boolean combination of subspaces.
Then $H = (B_1 \setminus C_1) \cup \cdots \cup (B_n \setminus C_n)$ where each $B_i$ is a coset of a subspace $W_i$ of $V$, each $C_i$ is a finite union of cosets of subspaces of $W_i$, and each $B_i \setminus C_i$ is nonempty.
We first show that $H = B_1 \cup \cdots \cup B_n$.
It suffices to show that $H$ contains each $B_i$.
For each $i$ let $D_i = B_i \setminus C_i$ and fix $w_i \in D_i$.
Fix $i$.
As $H$ is a subgroup $(D_i - w_i) - (D_i - w_i) \subseteq H$.
Now $B_i - w_i = W_i$, so by the claim $(D_i - w_i) - (D_i - w_i) = W_i$.
Finally $B_i \subseteq H$ as $w_i \in H$.

\medskip
Now each $w_i$ is in $H$, hence each $W_i$ is a subgroup of $H$.
As $H$ is a finite union of cosets of the $W_i$ another application of Neumann's lemma shows that some $W_i$ has finite index in $H$.
Fix $i$ with this property.
Finally, $H/W_i$ is a finite subgroup of $V/W_i$ and $V/W_i$ is torsion free, hence $H/W_i$ is trivial and so $H = W_i$.
\end{proof}

\begin{lemma}\label{lem:fisher}
Let $\F$ be a field, $V$ be an $\F$-vector space, and $\Sa V$ be an expansion of $V$ by subspaces of the $V^n$.
Then every $\Sa V$-definable set is a boolean combination of affine subspaces.
If $\F$ has characteristic zero then any $\Sa V$-definable subgroup of any $V^n$ is a subspace.
\end{lemma}

\begin{proof}
By Lemma~\ref{lem:vect bool} it is enough to prove the first claim.
Any subspace of $V^n$ is a finite intersection of kernels of linear maps $V^n \to V$.
Furthermore any linear map $V^n \to V$ is of the form $(v_1, \ldots, v_n) \mapsto f_1(v_1) + \cdots + f_n(v_n)$ where each $f_i$ is a linear map $V \to V$.
Hence we may suppose that $\Sa V$ is an expansion of $V$ by linear maps $V \to V$.
Equivalently we may suppose that $\Sa V$ is a module over an $\F$-algebra.
By the pp-quantifier elimination for modules any $\Sa V$-definable subset of $V^n$ is a boolean combination of cosets of pp-definable subgroups~\cite[Cor.~2.13]{Prest_modules}.
Finally, any pp-definable subgroup is a vector subspace~\cite[Lemma~2.1]{Prest_modules}.
\end{proof}

\begin{proposition}\label{prop:new F}
Let $\F$ be a field, $V$ be an infinite $\F$-vector space, and $\Sa V$ be an expansion of $V$ by subspaces of the $V^n$.
Then $\Sa V$ is locally trace definable in $\vvec_\F$.
Furthermore:
\begin{enumerate}[leftmargin=*]
\item Any module over an $\F$-algebra is locally trace definable in $\vvec_\F$.
\item If $\F$ has positive characteristic then any one-based expansion of an $\F$-vector space is locally trace definable in $\vvec_\F$.
\end{enumerate}
\end{proposition}

\begin{proof}
We first prove (2).
By Fact~\ref{fact:vec} it suffices to treat the case $\F = \F_p$.
Let $\Sa W$ be a one-based expansion of an $\F_p$-vector space $W$.
By Proposition~\ref{prop:abelai} and Fact~\ref{fact:disjoint union} it is enough to show that $W^n/B$ is locally trace definable in $\Th(W)$ for any definable subgroup $B$ of $W^n$.
Any subgroup of an $\F_p$-vector space is a subspace, hence $W^n/B$ is also an $\F_p$-vector space, so $W^n/B$ is either finite or elementarily equivalent to $W$.

\medskip
Note that (1) follows from the first claim as any module over an $\F$-algebra is interdefinable with an $\F$-vector space $W$ equipped with a family of subspaces of $W^2$.
We finally prove the first claim.
By (2) we may suppose that $\F$ is characteristic zero.
As in (2) it suffices to show that any definable subgroup of $V^n$ is a subspace.
This follows by Lemma~\ref{lem:fisher}.
\end{proof}

\begin{proposition}\label{prop:F}
Let $\F$ be a prime field and $\Sa O$ be an arbitrary structure.
Then the following are equivalent.
\begin{enumerate}[leftmargin=*]
\item $\Sa O$ is locally trace definable in $\vvec_\F$.
\item $\Sa O$ is trace definable in an $R$-module for some $\F$-algebra $R$.
\item $\Sa O$ is locally trace definable in an $R$-module for some $\F$-algebra $R$.
\end{enumerate}
Furthermore if $\F = \F_p$ then the following are equivalent.
\begin{enumerate}[leftmargin=*]\setcounter{enumi}{3}
\item $\Sa O$ is locally trace definable in $\vvec_\F$.
\item $\Sa O$ is trace definable in a one-based expansion of an $\F$-vector space.
\item $\Sa O$ is locally trace definable in a one-based expansion of an $\F$-vector space.
\end{enumerate}
\end{proposition}

\begin{proof}
Proposition~\ref{prop:dkaplin} shows that (1) implies (2), (2) clearly implies (3), and Proposition~\ref{prop:new F}(1) shows that (3) implies (1).
Furthermore (4) implies (5) as any module is one-based, (5) clearly implies (6), and (6) implies (4) by Proposition~\ref{prop:new F}(2).
\end{proof}

We now prove Proposition~\ref{prop:dkaplin}(3).

\begin{proof}
Let $\F$ be a countable field, $V$ be an infinite-dimensional $\F$-vector space, and $f$ be a linear surjection $V \to V^n$ for some $n \ge 2$.
Proposition~\ref{prop:new F} shows that $(V, f)$ is locally trace equivalent to $V$.
By Proposition~\ref{prop;ginsep} $\Th(V, f)$ is trace equivalent to $D^{\aleph_0}(\Th(V, f))$.
Hence $\Th(V, f)$ is trace equivalent to $D^{\aleph_0}(\vvec_\F)$ by Proposition~\ref{prop:key}(4). 
\end{proof}




We now consider local trace definability in abelian groups.

\begin{proposition}
\label{prop;oneb}
The following are equivalent for any structure $\Sa O$ in a language of cardinality at most $\kappa$.
\begin{enumerate}[leftmargin=*]
\item $\Sa O$ is locally trace definable in an abelian group.
\item $\Sa O$ is trace definable in a module.
\item $\Sa O$ is trace definable in a $\Z[x_i]_{i<\kappa}$-module.
\item $\Sa O$ is trace definable in an abelian structure.
\item $\Sa O$ is locally trace definable in an abelian structure.
\end{enumerate}
\end{proposition}

\begin{proof}
It is clear that (3) implies (2), (2) implies (4), and (4) implies (5).
A variation of the proof of Proposition~\ref{prop:dkaplin} shows that (1) implies (3), we leave the details of this to the reader.
It remains to show that (5) implies (1).
It is enough to show that any abelian structure is locally trace definable in an abelian group.
By Proposition~\ref{prop:abelai} and Fact~\ref{fact:disjoint union} it is enough to show that for any family $(A_i)_{i \in I}$ of abelian groups there is an abelian group $B$ which trace defines each $A_i$.
Let $B$ be the direct sum of the $A_i$ and apply the fact that an abelian group trace defines any summand\footnote{This is proven by applying Fact~\ref{fact:trace embedd} and the classical fact that abelian groups admit quantifier elimination after unary relations defining the set of $k$th multiples for all $k \ge 1$ are added to the language.} \cite[Lemma~4.9]{trace2}.
\end{proof}

We finally apply Lemma~\ref{lem:one un} to the theory $\doag$ of $(\R;+,<,0,1)$.
We consider a theory introduced by Block Gorman, Caulfield, and Hieronymi~\cite{patho}.
Let $\Q(t)$ be the field of rational functions over $\Q$ in the variable $t$.
Furthermore let $T_\mathrm{vd}$ be the theory of structures of the form $\Sa M=\left(M; +, <, (x\mapsto qx)_{q\in\Q(t)}, 0, 1 \right)$ satisfying the following.
\begin{enumerate}[leftmargin=*]
\item $(M;+,<,0,1)\models\doag$.
\item $\left(M;+,(x\mapsto qx)_{q\in\Q(t)}\right)$ is a $\Q(t)$-vector space.
\item If $q_1,\ldots,q_m \in \Q(t)$ are $\Q$-linearly independent then $\{ (q_1 a,\ldots,q_m a) : a \in M \}$ is dense in $M^m$.
\end{enumerate}
Note that such $\Sa M$ is interdefinable with $(M;+,<,x\mapsto tx)$.
Informally, $T_\mathrm{vd}$ is the theory of a divisible ordered abelian group $(M;+,<)$ equipped with a generic $\Q$-linear bijection $M\to M$.
By \cite[Prop.~3.7 and Thm.~3.14]{patho} $T_\mathrm{vd}$ admits quantifier elimination and is $\nip$.
Note that $T_\mathrm{vd}$ is not complete.

\begin{proposition}\label{prop;vd}
Every completion of $T_\mathrm{vd}$ is trace equivalent to $D^{\aleph_0}(\doag)$.
\end{proposition}


\begin{proof}
Let $\Sa M = (M; +, <, x \mapsto tx, 0, 1) \models T_\mathrm{vd}$ be $\aleph_1$-saturated.
We apply Lemma~\ref{lem:one un} with $T = \doag$ and $T^* = \Th(\Sa M)$.
Note that (1) follows by quantifier elimination for $T_\mathrm{vd}$.
Furthermore (2) follows as $1, t, ,\ldots,t^m$ are linearly independent over $\Q$ and hence $\{(a, ta, t^2 a,\ldots,t^m a) : a \in M\}$ is dense in $M^{m + 1}$ for each $m \ge 1$.
Let $V$ be the convex hull of $\Z$ in $M$, let $\mfrak$ be the set of $a \in M$ such that $|a| < 1/n$ for all $n \ge 1$, and let $E$ be the equivalence relation on $M$ given by declaring $E(a,a')$ if and only if $a - a' \in \mfrak$.
Both $V$ and $\mfrak$ are convex and hence definable in $(M;+,<)^\mathrm{Sh}$.
Furthermore $V/E$ is canonically identified with $\R$ and the structure induced on $\R$ by $(M;+,<)^\mathrm{Sh}$ is an expansion of $\rgoup$.
\end{proof}

\section{Fields and expansions of fields}
We first consider non-perfect fields.

\begin{proposition}
\label{prop;insep}
Suppose that $\Sa K$ is an expansion of a non-perfect field $K$.
Then $\Th(\Sa K)$ is trace equivalent to $D^{\aleph_0}(\Th(\Sa K))$.
\end{proposition}

If $K$ has characteristic $p$ and $\lambda \in K$ is not a $p$th power then the map $K^2\to K$ given by $(a,b)\mapsto a^p+\lambda b^p$ is a definable injection.
So Proposition~\ref{prop;insep} follows from Proposition~\ref{prop;ginsep}.

\medskip
Given a prime $p$ and a natural number $e \ge 1$ we let $\mathrm{SCF}_{p,e}$ be the theory of separably closed fields $K$ of characteristic $p$ such that $K$ has degree $p^e$ over the subfield of $p$th powers.
We also let $\acvfp$ be the theory of an algebraically closed field of characteristic $p$ equipped with a non-trivial valuation and let $\scvfpe$ be the theory of a model of $\scfpe$ equipped with a non-trivial valuation, see \cite{scvhf} for background on $\scvfpe$.

\begin{proposition}
\label{prop;seppp}
Fix a prime $p$ and $e\in\N_{\ge 1}$.
Then $\mathrm{SCF}_{p,e}$ is trace equivalent to $D^{\aleph_0}(\acf_p)$ and $\scvfpe$ is trace equivalent to $D^{\aleph_0}(\acvfp)$.
\end{proposition}

\begin{proof}
Proposition~\ref{prop;insep} shows that $\mathrm{SCF}_{p,e}$ is trace equivalent to $D^{\aleph_0}(\mathrm{SCF}_{p,e})$ and $\scvfpe$ is trace equivalent to $D^{\aleph_0}(\scvfpe)$.
Hence by Proposition~\ref{prop:key}(4) it is enough to show that $\mathrm{SCF}_{p,e}$ is locally trace equivalent to $\acf_p$ and $\scvfpe$ is locally trace equivalent to $\acvfp$.
Fix an $\aleph_1$-saturated model $(K, v)$ of $\scvfpe$, so $K \models \scfpe$.
Let $F$ be the subfield of $K$ consisting of elements which are $p^n$th powers for every $n \ge 1$.
Then $F \models \acf_p$, hence $F$ admits quantifier elimination, and so $K$ trace defines $F$ by Fact~\ref{fact:trace embedd}.
By \cite[Prop.~2.8]{scvhf} $F$ is dense in the $v$-topology on $K$, hence $(F, v) \models \acvfp$.
It follows that $(K, v)$ trace defines $(F, v)$ by quantifier elimination for $\acvfp$.
Now let $\kalg$ be the algebraic closure of $K$.
We show that $\kalg$ locally trace defines $K$
By \cite[Cor.~4.3]{messmer_marker_messmer_pillay_2017} there is a countable collection $\Cal E$ of functions $K \to K$ such that every formula in $K$ in the variables $x_1 ,\ldots, x_n$ is equivalent to a boolean combination of equalities between terms of the form $f(x_1, \ldots, x_n, h_1(x_{i_1}), \ldots, h_m(x_{i_m}))$ for $f$ a polynomial in $K[y_1, \ldots, y_{n + m}]$, $h_1,\ldots, h_m \in \Cal E$, and $i_1, \ldots, i_m \in \nset$.
Consider each $h \in \Cal E$ to be a function $K \to \kalg$.
Now $\Cal E$ and the inclusion $K \to \kalg$ witness local trace definability of $K$ in $\kalg$.
Finally, let $w$ be some valuation on $\kalg$ extending $v$.
A similar argument applying \cite[Lemma~2.10]{scvhf} shows that $(\kalg, w)$ locally trace defines $(K, v)$.
\end{proof}

We give field-theoretic analogues of Propositions~\ref{prop:dkaplin} and \ref{prop;vd}.
Given $\kappa\ge 1$ let $\dcf^\kappa$ be the model companion of the theory of a characteristic zero field equipped with $\kappa$ commuting derivations.
We also let $\ncdf^\kappa$ be the model companion of the theory of a characteristic zero field equipped with $\kappa$ (possibly non-commuting) derivations.
In particular $\dcf^1 = \ncdf^1$ is the theory of differentially closed fields of characteristic zero.
(Of course $\dcf^1$ is usually referred to as $\dcf_0$.)
We also let $\mathrm{CODF}^\kappa$ be the model companion of the theory of an ordered field equipped with $\kappa$ commuting derivations.

\begin{proposition}\label{prop:dcf}
\hspace{.00000000000000000000000000000000000000000000000000000000000001cm}
\begin{enumerate}[leftmargin=*]
\item If $\kappa \ge \aleph_0$ then $\dcf^\kappa$ is trace equivalent to $D^\kappa(\acf_0)$.
\item If $\kappa \ge 2$ then $\ncdf^\kappa$ is trace equivalent to $D^{\kappa + \aleph_0}(\acf_0)$.
\item If $\kappa \ge 1$ then $\mathrm{CODF}^\kappa$ is trace equivalent to $D^{\kappa + \aleph_0}(\rcf)$.
\end{enumerate}
\end{proposition}

As in Proposition~\ref{prop;triv} the bounds on $\kappa$ are sharp as $\dcf^\kappa$ is totally transcendental when $\kappa$ is finite.
Proposition~\ref{prop:dcf} follows by the same argument as the previous propositions and some standard facts about derivations.
However, we prove Proposition~\ref{prop:dcf} as a special case of a more general result in terms of generalizations of these theories that we first recall.

\medskip
Let $L$ be a language expanding the language of rings and $T$ be an $L$-theory expanding the theory of fields.
Then $T$ is {\bf algebraically bounded} if for every $\Sa K \models T$ and definable $X \subseteq K^m \times K$ there are polynomials $f_1,\ldots,f_n \in K[x_1,\ldots,x_m,t]$ such that if $\alpha\in K^m$ and $X_\alpha = \{\beta \in K : (\alpha,\beta) \in X \}$ is finite then $X_\alpha\subseteq\{ \beta \in K : f_i(\alpha,\beta) = 0 \}$ for some $i \in \nset$ such that $f_i(\alpha, x) \in K[x]$ is not constant zero.
Note that an algebraically bounded theory is geometric.
An expansion of a field is algebraically bounded when its theory is.
Examples of algebraically bounded structures include algebraically closed fields, real closed fields, and characteristic zero fields equipped with non-trivial henselian valuations~\cite{lou-dimension}.
In particular any characteristic zero henselian field is bounded.
(Recall that a field is henselian if it admits a non-trivial henselian valuation.)

\medskip
We also consider derivations on o-minimal structures.
Let $L$ be a language expanding the language of ordered fields, $T$ be an o-minimal $L$-theory expanding $\rcf$, and $\Sa R$ be a model of $T$.
Recall that a derivation $\der\colon R\to R$ is {\bf $T$-compatible} if 
$$\der(f(a_1,\ldots,a_n)) = \sum_{i=1}^n \der(a_i)\frac{\partial f}{\partial x_i}(a_1,\ldots,a_n)$$
for any parameter-free definable $C^1$-function $f\colon R^n\to R$ and $(a_1,\ldots,a_n)\in R^n$.

\begin{fact}\label{fact:comb der}
Let $L$ be a language expanding the language of rings and $T$ be an $L$-theory expanding the theory of characteristic zero fields.
Fix $\kappa\ge 1$ and let $L_\der$ be the expansion of $L$ by unary functions $\der_i$ for $i<\kappa$.
\begin{enumerate}[leftmargin=*]
\item If $T$ is algebraically bounded then the $L_\der$-theory of a model of $T$ equipped with $\kappa$ commuting derivations has a model companion relative to $T$ which we call $\dcf^\kappa_T$.
Furthermore $\dcf^\kappa_T$ is complete and admits quantifier elimination relative to $T$.
\item If $T$ is algebraically bounded then the $L_\der$-theory of a model of $T$ equipped with $\kappa$ (possibly non-commuting) derivations has a model companion relative to $T$ which we call $\ncdf^\kappa_T$.
Furthermore $\ncdf^\kappa_T$ is complete and admits quantifier elimination relative to $T$. 
\item If $T$ is o-minimal then the $L_\der$-theory of a model of $T$ equipped with $\kappa$ commuting $T$-compatible derivations has a model companion relative to $T$ which we call $\dcf^\kappa_T$.
Furthermore $\dcf^\kappa_T$ admits quantifier elimination relative to $T$.
\end{enumerate}
\end{fact}

We also let $\dcf_T = \dcf^1_T$.
We have used $\dcf_T^\kappa$ to denote two different theories.
One can apply cell decomposition to show that any algebraically bounded o-minimal expansion of an ordered field is interdefinable with the underlying field.
Furthermore if $R$ is a real closed field then any derivation $R \to R$ is $T$-compatible \cite[Prop.~2.8]{Fornasiero_2020}, so our two definitions agree when they overlap.
Note that $\dcf^\kappa$ is $\dcf^\kappa_T$ for $T = \acf_0$ and $\mathrm{CODF}^\kappa$ is $\dcf^\kappa_T$ for $T = \rcf$.

\medskip
Here (1) and (2) are due to Fornasiero and Terzo~\cite[Thm.~5.12 and Thm~4.5]{antongiulio2024} and (3) is due to Fornasiero and Kaplan~\cite[Thm.~4.8 and Lemma~4.11]{Fornasiero_2020}.
Actually, they only construct these theories in the case when $\kappa$ is finite.
But the generalization is immediate: let $\dcf^\kappa_T$ be the $L_\der$-theory such that the $L\cup\{\der_{i_1},\ldots,\der_{i_m}\}$-reduct of $\dcf^\kappa_T$ is $\dcf^m_T$ modulo the obvious relabeling for any distinct $i_1, \ldots, i_m < \kappa$, likewise for $\ncdf^\kappa_T$.

\medskip
In Theorem~\ref{thm;tcf} we prove a result about the theories described in Fact~\ref{fact:comb der}.
We first recall some necessary background.
We begin with a basic fact about derivations, see for example~\cite[Cor.~1.9.4 and Lemma~1.9.2]{trans}.

\begin{fact}\label{fact;der}
Suppose that $F/K$ is an extension of characteristic zero fields and $H \subseteq F$  is algebraically independent over $K$.
For any derivation $\der \colon K \to K$ and function $ f\colon H \to F$ there is a derivation $\der^* \colon F \to F$ such that $\der^*$ agrees with $\der$ on $K$ and agrees with $f$ on $H$.
\end{fact}

We will apply the case of Facts~\ref{fact;der}  when $\der$ is the constant zero derivation.
We need two facts about henselian fields.
Fact~\ref{fact:spine} is proven in \cite{growing_spines}.

\begin{fact}\label{fact:spine}
Any characteristic zero henselian field $K$ is elementarily equivalent to a field which admits a non-trivial henselian valuation with residue field $K$.
\end{fact}

Fact~\ref{fact:ext val} is due to Jahnke~\cite{jahnke-when}.

\begin{fact}\label{fact:ext val}
If $K$ is a $\nip$ field and $v$ is a non-trivial henselian valuation on $K$ with non-separably closed residue field then $v$ is definable in $K^\mathrm{Sh}$.
\end{fact}

We also need some basic facts about tame pairs of o-minimal structures for Theorem~\ref{thm;tcf}(4).
Let $T$ be an o-minimal theory expanding $\rcf$.
A {\bf tame pair} of models of $T$ is a model $\Sa R$ of $T$ equipped with a unary relation defining a proper elementary submodel $\Sa S$ of $\Sa R$ such that $\{s \in S : s < r\}$ has a supremum in $S \cup \{\pm \infty\}$ for every $r \in R$.
Note that if $\Sa R$ is a proper elementary extension of an o-minimal expansion of $\rfield$ then the expansion of $\Sa R$ by a unary relation defining $\R$ is a tame pair.
The theory of tame pairs of models of $T$ is complete~\cite{t-convexity}.
Fact~\ref{fact:tame pair} is \cite[Thm.~A]{t-convexityII}.

\begin{fact}\label{fact:tame pair}
Suppose that $(\Sa R, \Sa S)$ is a tame pair of models of an o-minimal theory $T$ expanding $\rcf$.
Let $V$ be the convex hull of $S$ in $R$, $\mfrak$ be the set of $r \in R$ such that $|r| \le |s|$ for all positive $s \in S$, and identify $V/\mfrak$ with $S$ in the obvious way.
Then the structure induced on $V/\mfrak$ by $(\Sa R, V)$ is interdefinable with $\Sa S$.
\end{fact}

Proposition~\ref{prop:dcf} is a special case of Theorem~\ref{thm;tcf}.

\begin{thm}
\label{thm;tcf}
Suppose that $L$ expands the language of rings, $T$ is an $L$-theory expanding the theory of characteristic zero fields, and $\kappa \ge 1$.
\begin{enumerate}
\item If $T$ is algebraically bounded and $\kappa \ge \aleph_0$ then $\dcf^\kappa_T$ is trace equivalent to $D^{\kappa}(T)$.
\item If $T$ is algebraically bounded and $\kappa \ge 2$ then $\ncdf^\kappa_T$ is trace equivalent to $D^{\kappa + \aleph_0}(T)$.
\item If $T$ is the theory of a non-algebraically closed characteristic zero $\nip$ henselian field then $\dcf^\kappa_T$ is trace equivalent to $D^{\kappa + \aleph_0}(T)$.
\item If $T$ is o-minimal then any completion of $\dcf^\kappa_T$ is trace equivalent to $D^{\kappa + \aleph_0}(T)$.
\end{enumerate}
\end{thm}

We say ``any" completion in (4) as $\dcf^\kappa_T$ may not be complete.
It is conjectured that any infinite $\nip$ field is elementarily equivalent to a henselian field and that any $\infty$-$\nip$ field is $\nip$.
By the proof of Theorem~\ref{thm:distinct} $D^{\aleph_0}(T)$ is trace maximal when $T$ is $\infty$-$\ip$.
Hence if both of these conjectures hold then (3) above covers all theories $T$ of non-algebraically closed characteristic zero fields for which $D^{\kappa + \aleph_0}(T)$ is not trace maximal.

\begin{proof}
We first prove (1) and the case of (2) when $\kappa \ge \aleph_0$.
By Lemma~\ref{lem:dkap} and Fact~\ref{fact:comb der} it suffices to suppose that $A$ is a subset of a characteristic zero field $K$ and $(f_i)_{i < \kappa}$ is a collection of functions $A \to K$ such that $A$ is algebraically independent over $\bigcup_{i < \kappa} f_i(A)$ and show that each $f_i$ extends to a derivation $\der_i \colon K \to K$ such that the $\der_i$ commute.
Let $F$ be the algebraic closure in $K$ of the subfield generated by $\bigcup_{i < \kappa} f_i(A)$, so $A$ is algebraically independent over $F$.
Let $B \subseteq K$ be such that $A \cup B$ is a transcendence basis for $F/K$.
For each $i < \kappa$ let $\der_i$ be a derivation $K \to K$ such that $\der_i$ agrees with $f_i$ on $A$, $\der_i$ vanishes on $B$, and $\der_i$ vanishes on $F$.
Note that $\der_i \circ \der_j$ is constant zero for any $i, j  < \kappa$, hence the $\der_i$ commute.

\medskip
The case of (2) when $2 \le \kappa < \aleph_0$ follows by applying Lemma~\ref{lem:two un} with $T^* = \ncdf^\kappa_T$ and $T^*_0$ the theory of a model of $T$ equipped with $\kappa$ derivations.

\medskip
We now suppose that $T$ is as in (3).
The case when $\kappa$ is infinite follows from (1) as $T$ is algebraically bounded.
Hence we may suppose that $\kappa = 1$.
We verify the conditions of Lemma~\ref{lem:one un}.
As the models of $T$ are not separably closed a theorem of Prestel-Ziegler shows that $T$ admits a definable topology which agrees with the henselian topology on any model of $T$ that admits a non-trivial henselian valuation~\cite{Prestel1978}.
We work with respect to this definable topology.
Fix a highly saturated model $(F, \der)$ of $\dcf_T$.
Let $m \ge 1$ and $X = \{ (a, \der(a), \ldots, \der^{(m)}(a)) : a \in F \}$.
Fix nonempty open $U_0, \ldots, U_m \subseteq F$.
Then $X$ intersects $U_0 \times \cdots \times U_m$ by \cite[Def.~3.3]{antongiulio2024} as $U_0 \times \cdots \times U_m$ is $(m + 1)$-dimensional.
Hence $X$ is dense in $F^{m + 1}$.
By Fact~\ref{fact:spine} and saturation there is a non-trivial Henselian valuation $v$ on $F$ such that the residue field of $v$ is elementarily equivalent to $F$.
Let $V$ be the valuation ring of $v$, $\mfrak$ be the maximal idea of $V$, and $E$ be the equivalence relation on $F$ given by declaring $E(a, a^*)$ when $a - a^* \in \mfrak$.
In particular, two elements of $V$ are $E$-equivalent if and only if they have the same residue.
Finally, by Fact~\ref{fact:ext val} both $V$ and $E$ are definable in $F^\mathrm{Sh}$.

\medskip
We finally prove (4).
Suppose that $T$ is o-minimal.
Note that any model of $\dcf^\kappa_T$ is trace definable in $D^{\kappa + \aleph_0}(T)$ by the same argument as above.
It is enough to fix a $|T|^+$-saturated model $(\Sa R, \dik)$ of $\dcf^\kappa_T$ and show that $\Th(\Sa R, \dik)$ trace defines $D^{\kappa + \aleph_0}(T)$.
By saturation there is an elementary substructure $\Sa S$ of $\Sa R$ such that $(\Sa R, \Sa S)$ is a tame pair of models of $T$.
Let $V, \mfrak$ be as in Fact~\ref{fact:tame pair} and let $E$ be the equivalence relation on $R$ given by declaring $E(a,a^*)$ when $a - a^* \in \mfrak$.
Note that $V$ and $\mfrak$ are both definable in $\Sh R$ as they are convex.
By Fact~\ref{fact:tame pair} the structure induced on $V/E$ by $\Sh R$ is interdefinable with an expansion of a model of $T$.
By \cite[Lemma~5.5]{Fornasiero_2020} $\{ (\der_{i_1}(a), \ldots, \der_{i_n}(a)) : a \in R\}$ is dense in $R^{n}$ for any distinct $i_1, \ldots, i_n < \kappa$ and $\{ (a, \der_i(a), \ldots, \der^{(m)}_i(a)) : a \in R \}$ is dense in $R^{m + 1}$ for each $m \ge 1$ and $i < \kappa$.
Apply Lemma~\ref{lem:one un}.
\end{proof}

We give a corollary to dense pairs.
Recall that a {\bf dense pair} of o-minimal structures is an expansion of an o-minimal structure by a unary relation defining a dense proper elementary substructure.
In particular if $\Sa R$ is an o-minimal expansion of $\rfield$ then the expansion of $\Sa R$ by a unary relation defining any proper elementary substructure is a dense pair.
We let $\tdense$ be the theory of dense pairs of models of an o-minimal theory $T$ expanding $\rcf$, this is complete by~\cite{dense-pairs}.

\begin{corollary}\label{cor:dense pair}
Let $T$ be an o-minimal theory expanding $\rcf$.
Then $\tdense$ is locally trace equivalent to $T$, trace definable in $D^{\aleph_0}(T)$, but does not trace define $D^{\aleph_0}(T)$.
\end{corollary}

Recall that the {\bf constant subfield} of a differential field $(F, \der)$ is $\{a \in  F : \der(a) = 0 \}$.

\begin{proof}
We first show that $\tdense$ is locally trace equivalent to $T$ and trace definable in $D^{\aleph_0}(T)$.
As $T$ is a reduct of $\tdense$ it suffices to show that $\tdense$ is trace definable in $D^{\aleph_0}(T)$.
The constant subfield of any $(\Sa R, \der) \models \dcf_T$ is a dense proper elementary substructure of $\Sa R$ by the proof of~\cite[Cor.~4.22]{Fornasiero_2020}.
Hence $\tdense$ is interpretable in $\dcf_T$.
Apply Theorem~\ref{thm;tcf}(4).

\medskip
We now show that $\tdense$ does not trace define $D^{\aleph_0}(T)$.
Berenstein, Dolich, and Onshuus showed that $\tdense$ is strongly dependent~\cite{gdensepairs}.
Strong dependence is preserved under trace definability by Fact~\ref{fact:preserve}.
So it suffices to show that $D^{\aleph_0}(T)$ is not strongly dependent.
In fact, it is easy to see that $D^{\aleph_0}(\triv)$ is not strongly dependent.
\end{proof}

We finally show that Proposition~\ref{prop;triv} is sharp.
We just give a sketch of the proof.

\begin{lemma}\label{lem:A fin}
If $n \ge 1$ then $A_n$ is totally transcendental.
\end{lemma}

\begin{proof}
Fix $(M; f_1, \ldots, f_n) \models A_n$.
An easy application of Fact~\ref{fact;der} shows that there is an embedding $\uptau$ of $(M; f_1, \ldots, f_n)$ into the $\{\der_1, \ldots, \der_n\}$-reduct of some model of $\dcf^n$.
By Fact~\ref{fact:trace embedd} $(M, f_1, \ldots, f_n)$ is trace definable in $\dcf^n$.
Finally, $\dcf^n$ is totally transcendental by a theorem of McGrail~\cite{macgrail}, so $A_n$ is totally transcendental by Fact~\ref{fact:preserve}.
\end{proof}

\subsection{More theories trace equivalent to $D^{\aleph_0}(\rcf)$}\label{section:rcf}
We have shown that $D^{\aleph_0}(\rcf)$ is trace equivalent to the theory of a real closed field equipped with a generic derivation.
We now show that it is also trace equivalent to the theory of a more interesting ordered differential field.
See \cite{trans} for an account of the transseries.

\begin{proposition}\label{prop:rcf}
The theory of tame pairs of real closed fields and the theory of the ordered differential field of transseries are both trace equivalent to $D^{\aleph_0}(\rcf)$.
\end{proposition}

Combining with Corollary~\ref{cor:dense pair} we see that the theory of tame pairs of real closed fields trace defines the theory of dense pairs of real closed fields but not vice versa.

\begin{proof}
Let $(\T,\der)$ be the ordered differential field of transseries, so in particular $\T$ is a real closed field extending $\R$ and $\der$ is a derivation $\T \to \T$ with constant subfield $\R$.
Therefore $(\T,\der)$ interprets $(\T,\R)$, so the theory of tame pairs of real closed fields is interpretable in the transseries.
It follows directly from \cite[Lemma~16.5.5]{trans} that there is a real closed field $R$ extending $\T$ such that $\{ \der^{(n)} : n \in \N \}$, when considered as a collection of functions $\T \to R$, witnesses  local trace definability of $(\T, \der)$ in $R$.

\medskip
We show that the theory of tame pairs of real closed fields trace defines $D^{\aleph_0}(\rcf)$.
Let $\pusix$ be the field of Puiseux series over $\R$ in the variable $\varepsilon$.
Now $\pusix$ is a real closed field, hence $(\pusix,\R)$ is a tame pair of real closed fields.
We apply Lemma~\ref{lem;g}.
Let $\st$ be the usual standard part map $\pusix\to\R\cup\{\pm\infty\}$ and note that $\st$ is definable in $(\pusix,\R)$.
We  define a function  $g_n\colon \pusix\to \R\cup\{\pm\infty\}$ for each $n$ by applying induction on $n$.
Let $g_0=\st$.
Given $n\ge 1$ let $g_n(\alpha) = g_{n - 1}(\alpha)$  if $g_{n - 1}(\alpha) \in \{\infty,-\infty\}$ and otherwise let
\begin{align*}
g_n(\alpha) &= \st\left(\frac{\alpha - g_0(\alpha)-g_1(\alpha)\varepsilon-\cdots-g_{n-1}(\alpha)\varepsilon^{n-1}}{\varepsilon^n}\right)\\
&= \sup\{\beta\in\R : g_0(\alpha) + g_1(\alpha)\varepsilon+\cdots+g_{n-1}(\alpha) \varepsilon^{n-1} + \beta \varepsilon^n < \alpha\}.  
\end{align*}
It follows by applying induction that each $g_n$ is definable in $(\pusix,\R)$.
Now observe that if $(b_i)_{i\in\N}$ is any sequence of real numbers then
$$g_i\left(\sum_{m \in\N} b_m\varepsilon^m\right) = b_i \quad\text{for all $i\in \N$}.$$
An application of Lemma~\ref{lem;g} shows that $\Th(\pusix,\R)$ trace defines $D^{\aleph_0}(\rcf)$.
\end{proof}

\bibliographystyle{abbrv}
\bibliography{NIP}
\end{document}